\documentclass[12pt]{article}

\usepackage{amsmath,amsthm,amsfonts,amssymb,amscd}
\usepackage[latin1]{inputenc}
\usepackage{color}
\usepackage{psfrag}
\usepackage[dvips]{graphicx}
\usepackage{epsfig}

\newcommand{\interior}{\operatorname{int}}
\newcommand{\diff}{\operatorname{Diff}}
\newcommand{\supp}{\operatorname{supp}}

\newcommand{\Fix}{\operatorname{Fix}}
\newcommand{\Per}{\operatorname{Per}}
\newcommand{\spann}{\operatorname{span}}

\newcommand{\dist}{\operatorname{dist}}

\newcommand{\EE}{{\mathcal E}}
\newcommand{\CC}{{\mathcal C}}
\newcommand{\FF}{{\mathcal F}}
\newcommand{\GG}{{\mathcal G}}

\newcommand{\KS}{{\mathcal{K}\mathcal{S}}}
\newcommand{\OO}{{\mathcal O}}

\newcommand{\RR}{{\mathcal R}}

\newcommand{\UU}{{\mathcal U}}
\newcommand{\VV}{{\mathcal V}}

\newcommand{\T}{{\mathbb T}}
\newcommand{\tM}{{\tilde{M}}}
\newcommand{\tFF}{{\tilde{\mathcal{F}}}}
\newcommand{\tC}{\tilde{\mathcal{C}}}
\newcommand{\tK}{\tilde{K}}
\newcommand{\tx}{{\tilde{x}}}
\newcommand{\ty}{{\tilde{y}}}
\newcommand{\tz}{{\tilde{z}}}
\newcommand{\tgamma}{{\tilde{\gamma}}}
\newcommand{\wtC}{{\widetilde{C}}}
\newcommand{\wtAC}{{\widetilde{AC}}}

\theoremstyle{plain}
\newtheorem{lemma}{Lemma}[section]
\newtheorem{corollary}[lemma]{Corollary}
\newtheorem{proposition}[lemma]{Proposition}
\newtheorem{theorem}[lemma]{Theorem}
\newtheorem{maintheorem}{Theorem}

\theoremstyle{definition}
\newtheorem{example}{Example}
\newtheorem{remark}[lemma]{Remark}
\newtheorem{definition}[lemma]{Definition}

\newtheorem{conjecture}{Conjecture}

\newcommand{\la}{\lambda}

\title{Stable Ergodicity and Accessibility for certain Partially
Hyperbolic Diffeomorphisms with Bidimensional Center Leaves}
\author{\sc Vanderlei Horita and Martin Sambarino
\thanks{Work partially supported by CAPES, FAPESP, PRONEX
and PROSUL, Brazil; Palis-Balzan project and CSIC, Dynamic Group 618,
Uruguay.}}
\date{}

\begin{document}

\maketitle \pagestyle{plain}

\begin{abstract}
We consider  classes of  partially hyperbolic diffeomorphism $f:M\to M$
with splitting $TM=E^s\oplus E^c\oplus E^u$ and $\dim E^c=2.$
These classes include for instance (perturbations of) the product of
Anosov and conservative surface diffeomorphisms, skew products of surface
diffeomorphisms over Anosov, partially hyperbolic symplectomorphisms
on manifolds of dimension four with bidimensional center foliation whose
center leaves are all compact.
We prove that accessibility holds in these classes for $C^1$ open and
$C^r$ dense subsets and moreover they are stably ergodic.
\end{abstract}


\section{Introduction}
\label{s.int}

Ergodicity plays a fundamental role in Dynamics (and in Probability and
Physics) since L. Boltzmann stated the ``ergodic hypothesis" which says
(roughly speaking) that in an evolution law time average and space
average are equal.
More precisely, we say that a dynamical system $f:M\to M$ preserving a
finite measure $m$ is ergodic (with respect to $m$) if any invariant set
has zero measure or its complement has zero measure.

E. Hopf \cite{Hop39} proved the ergodicity of the geodesic flow on
surfaces of negative curvature.
This was extended by Anosov to the geodesic flow on compact manifolds with
negative curvature in a cornerstone paper in dynamics \cite{An67}.
He also proved that conservative (today called) Anosov $C^{1+\alpha}$
diffeomorphisms are ergodic.
And, since Anosov diffeomorphisms are open, the above implies that
conservative Anosov systems are stably ergodic.
We say that a $C^r$ diffeomorphism $f:M\to M$ preserving a measure $m$ is
$C^r$ \textit{stably ergodic} if any sufficiently small $C^r$ perturbation
of $f$ preserving $m$ is ergodic.

In a seminal work, Grayson, Pugh, and Shub \cite{GPS94} proved that the
time one map of the geodesic flow of a hyperbolic surface is $C^2$ stably
ergodic.
Afterwards, Ch. Pugh and M. Shub recovered (in some sense) Smale´s
program in the sixties about stability and genericity by restricting to
partially hyperbolic diffeomorphisms on manifolds preserving the Lebesgue
measure and replacing structural stability by stable ergodicity.
They conjectured that among $C^2$  partially hyperbolic diffeomorphism
preserving the Lebesgue measure $m$, stable ergodicity holds in an open and
dense set.
They  proved important results in this direction and they proposed a
program as well (see \cite{PS97m}, \cite{PS97}, and \cite{PS00}).
The main conjecture is:

\begin{conjecture}[\cite{PS00}]
On any compact manifold, ergodicity holds for an open and dense set of
$C^2$ volume preserving partially hyperbolic diffeomorphisms.
\end{conjecture}

This conjecture splits into two conjectures where accessibility (see
Definition~\ref{defacc}) plays a key role:

\begin{conjecture}[\cite{PS00}]
Accessibility holds for an open and dense set of $C^2$ partially
hyperbolic diffeomorphism, volume preserving or not.
\end{conjecture}

\begin{conjecture}[\cite{PS00}]
A partially hyperbolic $C^2$ volume preserving diffeomorphism with the
essential accessibility property is ergodic.
\end{conjecture}

They also proved \cite{PS00} a result in the direction of the third
conjecture: \textit{A partially hyperbolic $C^2$ volume preserving
diffeomorphism, dynamically coherent, center bunched, and with the
essential accessibility property is ergodic}.
Since then, a lot of research on the field has been done.
See the surveys \cite{BPSW01}, \cite{RRU07}, \cite{Wi10}, and \cite{CrICM}
for an account on this progress during the last decades.

In \cite{BW10}, K. Burns and A. Wilkinson improved a lot Pugh-Shub result
in two directions: dynamically coherence is not needed and the center
bunching condition is much milder than originally stated.

The key fact thus to obtain ergodicity is \textit{accessibility}.
In \cite{DW03} it is proved that accessibility holds for a $C^1$ open and
dense subset of $C^r$ partially hyperbolic diffeomorphism, volume
preserving or not.
When the center bundle has dimension one, it is proved in \cite{RRU08}
that accessibility holds for a $C^1$ open and $C^r$ dense subset of $C^r$
partially hyperbolic volume preserving diffeomorphism (later extended to
the non-volume preserving case in \cite{BRRTU08}).
This in particular implies the main conjecture in its full generality when
the center dimension is one.

There has been in the last years a great advance to the main conjecture in
the $C^1$ topology.
In fact in \cite{RRTU07} it is proved that stably ergodicity is $C^1$
dense when the center dimension is two.
And recently, an outstanding  result has been obtained by A. Avila, S.
Crovisier, and A. Wilkinson \cite{ACW}: stable ergodicity is $C^1$ dense in
any case (without any assumption on the dimension of the center bundle).

These results depends heavily on perturbation techniques available in the
$C^1$ topology and not known on higher topologies.
The $C^r$ denseness of stable ergodicity, $r\ge 2,$ is a complete different
problem.
Little is known in this case when the center bundle has dimension greater
than one.
In \cite{BW99}, the authors prove $C^r$ density of stable ergodicity for
group extensions over Anosov diffeomorphisms.
A remarkable  result has been obtained by F. Rodriguez-Hertz \cite{RH05}
for certain automorphisms of the torus $\T^d$.
Also, in \cite{SW00a} are given two examples that can be $C^r$, $r\ge 2$,
approximated by stable ergodic ones.
And very recently Z. Zhang \cite{Zha15} obtained $C^r$ density of stable
ergodicity for volume preserving diffeomorphisms satisfying some pinching
condition and a certain type of dominated splitting on the center.
A. Avila and M. Viana have announced $C^1$ openness and $C^r$ density for
certain skew product of surfaces diffeomorphisms over Anosov and our work
might have some overlap with theirs although our methods are different.

Our aim in this paper is to contribute to the $C^r$ denseness of stable
ergodicity, in particular when the center dimension is two.
We prove that for large classes of $C^r$ partially hyperbolic volume
preserving diffeomorphisms with two dimensional center bundle, stable
ergodicity holds in $C^r$ dense subsets.
Precise statements are given in Section~\ref{ss.results}.
However, just to give a flavor of them let us state a particular case
(see Theorem~\ref{main}).

\begin{maintheorem}
\label{t.intro}
Ergodicity holds in $C^1$ open and $C^r$ dense subset in the following
settings:
\begin{itemize}
\item Skew products of conservative surfaces diffeomorphisms over
      conservative Anosov diffeomorphisms
\item Partially hyperbolic symplectomorphisms on $(M,\omega)$ where
      $dimM=4$ having a bidimensional center foliation whose leaves are
      all compact.
\end{itemize}
\end{maintheorem}

The  main tool we use to prove the ergodicity is accessibility.
Thus, we have to prove that accessibility holds in a $C^1$ open and $C^r$
dense subset in the setting we are working with.
The main idea is to use results on conservative surface dynamics to show
that generically one gets accessibility.
Indeed, when the center dimension is two and we look to the accessibility
class inside a (periodic) compact center leaf we have three possibilities:
it has zero, one or two topological dimensions.
We prove that generically (see Theorem~\ref{trivialaccesibility}) zero
dimensional accessibility classes do not exist.
We will use to the full extent results on conservative surface dynamics to
prove that also generically one-dimensional accessibility classes do not
exist and therefore the accessibility classes are open on the
center leaf and so there is just one accessibility class.

\subsection{Setting}
\label{ss.setting}

Let $f:M\to M$ be a diffeomorphism where $M$ is a compact riemannian
manifold without boundary.
We say that $f$ is {\em partially hyperbolic} if the tangent bundle splits
into three subbundles $TM=E^s\oplus E^c\oplus E^u$ invariant under the
tangent map $Df$ and such that:
\begin{itemize}
\item There exists $0<\la<1$ such that
$$
\|Df_{/E^s}\|<\la \qquad \mbox{ and } \qquad \|Df^{-1}_{/E^u}\|<\la.
$$
\item For every $x\in M$ we have
$$
\frac{\|Df_{/E^s_x}\|}{m\{Df_{/E^c_x}\}}<1 \qquad \mbox{ and } \qquad
\frac{\|Df_{/E^c_x}\|}{m\{Df_{/E^u_x}\}}<1,
$$
where $m\{A\}$ is the co-norm of $A,$ i.e.,
$m\{A\}=\|A^{-1}\|^{-1}.$
\end{itemize}
By continuity of $Df$ and the compactness of $M$, there is a positive
constant $\eta <1$ such that the inequalities in the last item hold for
$\eta$ instead of $1$.
In other words, $E^s$ is uniformly contracted, $E^u$ is uniformly
expanded and the behaviour of $E^c$ is between both.

It is well known that the subbundles $E^s$ and $E^u$ uniquely
integrate to two foliations $\FF^s=\FF^s_f$ and $\FF^u=\FF^u_f$
called the stable and unstable foliation respectively.
We denote by $\FF^\sigma(x)\,(\sigma=s,u)$ the leaf of the foliation
through the point $x.$

On the other hand it is not always true that the center subbundle
$E^c$ is integrable. We say that the partially hyperbolic
diffeomorphism $f$ is \textit{dynamically coherent} if the bundles
$E^s\oplus E^c$ and $E^c\oplus E^u$ integrate to invariant
foliations $\FF^{cs}$ and $\FF^{cu}$ called the center stable and
center unstable foliations respectively. In particular $E^c$
integrates to a (normally hyperbolic) invariant foliation $\FF^c.$
Moreover, $\FF^c$ and $\FF^s$ subfoliates $\FF^{cs}$ and $\FF^c$ and
$\FF^u$ subfoliates $\FF^{cu}$, see \cite{BW08}.

We say the center foliation is $r$-\emph{normally hyperbolic} ($r\ge 1$)
if the following holds:
$$
\frac{\|Df_{/E^s_x}\|}{m\{Df_{/E^c_x}\}^r}<1 \qquad \mbox{ and } \qquad
\frac{\|Df_{/E^c_x}\|^r}{m\{Df_{/E^u_x}\}}<1.
$$
If $f$ is of class $C^r$ and the center foliation is $r$-normally
hyperbolic then the leaves of $\FF^c$ are of $C^r$ class (see \cite{HPS77}).

Partially hyperbolic diffeomorphisms are $C^1$ open.
In order to assure that dynamically coherence also holds for $C^1$ systems
nearby we have to require \textit{plaque expansiveness}.
This is technical and we will not define it here, we refer to \cite{HPS77}
(however, if $\FF^c$ is a $C^1$ foliation or all leaves of $\FF^c$ are
compact then the center foliation is plaque expansive).
The results on \cite{HPS77} (see Theorem 7.4) assure that a normally
hyperbolic and plaque expansive foliation $\FF^c_f$  of a diffeomorphism
$f$ is structurally stable, that is, there exist a neighborhood of
$\UU(f)$ and a homeomorphism $h:M\to M$ such that for $g\in\UU(f)$
there exists a (normally hyperbolic) foliation $\FF^c_g$ such that
$h(\FF^c_f(x))=\FF^c_g(h(x))$ and $h(\FF^c_f(f(x)))=\FF^c_g(g(h(x))).$
This result implies that \textit{partially hyperbolic diffeomorphisms,
dynamically coherent with center foliation plaque expansive are $C^1$
open}.

We also say that a partially hyperbolic diffeomorphism $f$ is
\textit{center bunched} if\;:
$$
\|Df_{/E^s_x}\|\frac{\|Df_{/E^c_x}\|}{m\{Df_{/E^c_x}\}}<1 \qquad
\mbox{ and } \qquad \frac{\|Df_{/E^c_x}\|}{m\{Df_{/E^c_x}\}}
\frac{1}{m\{Df_{/E^u_x}\}}<1.
$$

This bunching condition is as in \cite{BW10} where they improve
substantially the one stated by Pugh-Shub originally.
Notice that the bunching condition is also $C^1$ open.

We say that a partially hyperbolic diffeomorphism $f:M\to M$ dynamically coherent has
\textit{Global Product Structure} if there is a covering
$\pi:\tilde{M}\to M$ and a lift $\tilde{f}:\tilde{M}\to \tilde{M}$ of $f$
such that when we lift the invariant foliations (stable, unstable,
center-stable, and center-unstable) to $\tilde{M}$ we have for any
$\tilde{x}, \tilde{y}$ in $\tilde{M}$:
$$
\#\{\tilde{\FF}^{cs}(\tilde{x})\cap\tilde{\FF}^u(\tilde{y})\}=1\,\,\,
\qquad \mbox{ and }\,\,\, \qquad
\#\{\tilde{\FF}^{cu}(\tilde{x})\cap\tilde{\FF}^s(\tilde{y})\}=1.
$$

Under the conditions where the center foliation is structurally stable we
have that Global Product Structure is also $C^1$ open.
The above allow us to define a global projection (in $\tilde{M}$)
onto a given center stable manifold along the holonomy of the
unstable foliation.

\begin{definition}
Let $M$  be a compact riemannian manifold without boundary and let $r\ge1$.
We denote by $\EE^r=\EE^r(M)$ the set of $C^r$ diffeomorphisms $f:M\to M$
(with the $C^r$ topology) such that
\begin{itemize}
\item $f$ is partially hyperbolic;
\item $f$ is dynamically coherent;
\item the center foliation is $r$-normally hyperbolic and plaque
expansive;
\item $f$ is center bunched;
\item $f$ has Global Product Structure; and
\item the set of center leaves that are compact and $f$-periodic are dense
in $M$.
\end{itemize}
\end{definition}

We remark that $\EE^r$ is $C^1$ open (and hence $C^r$ open as well).
\vskip 15pt
\noindent\textbf{Examples.} Here we give some examples of diffeomorphism
in $\EE^r.$
We restrict ourselves where the center dimension is two.
\begin{enumerate}
\item \textit{Perturbation of product of diffeomorphisms:}
Let $g:S\to S$ be a $C^r$ diffeomorphism of a compact surface and let
$f:N\to N$ be a transitive Anosov diffeomorphism.
If the contraction and expansion of $f$ are strong enough we get that
$f\times g \in \EE^r(M)$ where $M=N\times S$.
Notice that the center foliation consists of compact manifolds
homeomorphic to $S.$
In case $g=id$ then automatically $f\times g\in\EE^r$ for any $r\ge 1.$
Notice that if we denote by $\tilde{N}$ the universal covering of $N$ then
setting $\tilde{M} = \tilde{N} \times S$ and lifting $f \times g$ to
$\tilde{M}$ we have Global Product Structure.

\item Let $f:M\to M$ be a volume preserving partially hyperbolic
diffeomorphism dynamically coherent whose center leaves are all compact.
Results of Gogolev \cite{Go11} and Carrasco \cite{CaP} state that the
center foliation is uniformly compact (that is, the leaves have finite
holonomy).
And also, Bohnet \cite{Boh13} proved that if the center foliation is
uniformly compact and $\dim E^u$ is one, then there is a lift that fibers
over an Anosov map on a torus.
In particular we have that there is a lift having Global Product
Structure.
Also, since $f$ is volume preserving, the periodic leaves are dense
(see also \cite{BB}).
Thus, if $f:M\to M$ is a volume preserving partially hyperbolic
diffeomorphism, dynamically coherent, $r$-normally hyperbolic, whose
center leaves are all compact, two dimensional, and $\dim M \le 5$
(or $\dim E^s=1$ or $\dim E^u=1)$ then $f\in\EE^r.$

\item \textit{Skew Products over Anosov:}
Let $f:N\to N$ be a $C^r$ (transitive) Anosov diffeomorphism and consider
$S$ a compact surface.
Let $\mathcal U\subset \diff^r(S)$ be the open set such that if
$h\in \mathcal U$ then $f\times h$ is partially hyperbolic with center
fiber $\{x\} \times S$, center bunched, and $r$-normally hyperbolic.
Let $g:N\to \mathcal U$ be a continuous map.
For $x\in N$ lets denote by $g_x$ the diffeomorphism $g(x):S\to S$.
For such a map $g$ consider the skew product $F=f\times_{sp}g:N\times S
\to N\times S$ by
$$
F(x,y)=(f(x), g_x(y)).
$$
We have that $F\in \EE^r(N\times S).$
We may consider thus perturbations of $F$ in $\EE^r$ and also
perturbations in the skew product setting.
For this, let $\GG=\{g:N\to \UU:\mbox{continuos}\}$, where
$g,\tilde{g}\in\GG$ are close if $g_x,\tilde{g_x}$ are $C^r$ close for all
$x\in N.$
We denote by $\EE^r_{sp}$ the set of skew products $f\times_{sp}g$ with
$g\in\GG$.

If $\omega$ is an area (symplectic) form on $S$ we denote by
$\EE^r_{sp,\omega}$ the set of skew products as above where $g_x$
preserves $\omega$ for all $x\in N.$

\item \textit{(Perturbation of) the product of the time t of an Anosov
suspension and a rotation:}
Consider $f:N\to N$ the time $t$ map of the suspension of a transitive
Anosov diffeomorphism and let $R:\mathbb{S}^1\to \mathbb{S}^1$ be a
rotation.
Let $f\times R:N\times \mathbb{S}^1\to N\times \mathbb{S}^1$.
It is not difficult to see that belongs to $\EE^r$ for any $r$ as long as
$f$ is $C^r.$

\item \textit{(Perturbation of) the product of time maps of Anosov
suspensions:}
Let $f, g$ be time maps of the suspensions of a transitive Anosov
diffeomorphisms.
Then $f\times g$ belongs to $\EE^r.$
\begin{remark}
We considered time maps of Anosov suspensions so that there is a lift with
Global Product Structure.
There are time-$1$ map of Anosov flows without Global Product Structure,
for instance, time-$1$ map of the geodesic flow in a surface of negative
curvature.
\end{remark}
\end{enumerate}

\subsection{Statements of Results}
\label{ss.results}

We denote by $\EE^r_m(M)$ the set of diffeomorphisms in $\EE^r$ preserving
a volume form $m$ on $M$, and by $\EE^r_\omega$ the ones in $\EE^r$
preserving a symplectic form $\omega$ on $M$.
And recall that $\EE^r_{sp}(M)$, $\EE^r_{sp,\omega}(M)$ are the skew
products over Anosov diffeomorphism on $M=N\times S$ where $S$ is a
compact surface and $\omega$ is an area form on $S.$

Our results mainly concerns accessibility, so let us introduce the
concept.

\begin{definition}
\label{defacc}
Let $f:M\to M$ be a partially hyperbolic diffeomorphism in $\EE^r$.
A {\em $su$-path} is a continuous curve $\alpha:[0,1]\to M$ such that
there exists a partition $0=t_0<t_1< \dots <t_n=1$ such that
$\alpha([t_i,t_{i+1}])$ is contained either in a leaf of $\FF^s$ or in a
leaf of $\FF^u$.
The relation $x\sim y$ if there exists a $su$-path from $x$ to $y$ is an
equivalence relation on $M.$

For a point $x\in M$ the accessibility class $AC(x)$ of $x$ is:
$$
AC(x)=\{y\in M \colon \text{there is a } su\mbox{-path from }x
\mbox{ to }y\}.
$$

We say that $f$ is {\em accessible} if $AC(x)=M$ for some $x$ (and hence
for all $x\in M)$.
On the other hand, we say the accessibility class $AC(x)$ is
\emph{trivial} if $AC(x)\cap \FF^c(x)$ is totally disconnected.
\end{definition}

Our first result concerns trivial accessibility classes (with no
restriction on the dimension of the center leaves):

\begin{maintheorem} \label{trivialaccesibility}
Let $r\ge 2$ and let $\EE$ denote $\EE^r, \EE^r_m, \EE^r_\omega,
\EE^r_{sp}$ or $\EE^r_{sp,\omega}$.
Then, the set $\RR_0$ of diffeomorphisms in $\EE$ having no trivial
accessibility classes is $C^1$ open and $C^r$ dense.
\end{maintheorem}

The next result gives a condition to assure accessibility when the
center leaves have dimension two.
We say that $f\in\EE^r$ supports
a {\em Center Axiom-A} if there is a periodic compact center leave
$\FF^c_1$ such that $f^k_{/\FF^c_1}$ is an Axiom-A diffeomorphism
without having both periodic attractor and periodic repellers, where $k$
is the period of $\FF^c_1$.
We remark that the set $f\in\EE^r$ supporting a Center Axiom-A is open in
$\EE^r.$

\begin{maintheorem}\label{axa}
Let $r\ge 2$ and let $\EE_A$ (respectively $\EE_{A,m}$) be the set of
diffeomorphism in $\EE^r$ (respectively $\EE^r_m$)  supporting a Center
Axiom-A.
Assume that $\dim E^c=2$.
Then, the set of diffeomorphisms $\RR_1$ in $\EE_A$ (or $\EE_{A,m}$) that
are accessible are $C^1$ open and $C^r$ dense in $\EE_A$ (respectively
$\EE_{A,m}$).
In particular in $\EE_{A,m}$ they are stably ergodic.
\end{maintheorem}

Finally, our main result says that accessibility holds generically
when, for instance, a symplectic form is preserved (it implies
Theorem~\ref{t.intro}):

\begin{maintheorem}\label{main}
Let $r\ge 2$ and let $\EE$ denote $\EE^r_\omega$ or $\EE^r_{sp,\omega}$.
Then, there exists $\RR\subset\EE$ which is $C^1$ open and $C^r$ dense
such that all diffeomorphism in $\RR$ are accessible.
In particular, in the case $\EE=\EE^r_\omega$ they are stably ergodic with
respect to volume form induced by $\omega$ and in the skew product case
$\EE=\EE^r_{sp,\omega}$ over a conservative Anosov diffeomorphism, they
are stably ergodic with respect to the volume form induced by volume form
of the Anosov and $\omega.$
\end{maintheorem}

\textbf{Organization of the paper:} In Section~\ref{s.accessibility} we
give general facts concerning accessibility classes and some results
regarding its structure when the center has dimension two.
In Section~\ref{s.perturbation} we prove some perturbation results in
order  to obtain later some generic results on the accessibility classes.
Section~\ref{s.theorem-trivialaccesibility} is devoted to prove
Theorem~\ref{trivialaccesibility}.
The accessibility classes of periodic points are studied in
Section~\ref{s.pp}, where it is proved that generically, when the center
subbundle has dimension two, the accessibility classes of hyperbolic
periodic points or elliptic (when we restrict to an  invariant center
leaf) are open.
Theorem~\ref{axa} is proved in Section~\ref{s.axa}.
Finally, in Section~\ref{s.main} we prove Theorem~\ref{main}.

\smallskip

\noindent\textit{Acknowledgements:} We wish to thank C. Bonatti, P. Le
Calvez, A. Koropecki, E. Pujals, M. Viana for useful conversations and specially R. Potrie for
reading a draft version of the paper and giving some insightful comments.

\section{Basic facts on accessibility}
\label{s.accessibility}

In this section we establish some basic results on accessibility.
We assume that $f:M\to M$ belongs to $\EE^r, r\ge 2$, although some remarks
hold in general.

When $y\in\FF^s(x)$ denote by $\Pi^s_f(\FF^c(x),\FF^c(y))$ the {\em (local)
holonomy map} from a neighborhood of $x$ in $\FF^c(x)$ to a neighborhood
of $y$ in $\FF^c(y)$ along the stable leaves (inside $\FF^{cs}(x) =
\FF^{cs}(y))$.
This map is well defined since the leaves of $\FF^s$ are simple connected.
By \cite{PSW97} this holonomy map is of class $C^1,$ i.e., the holonomy
map inside center stable leaves along stable leaves is $C^1.$
The same holds for center unstable leaves and holonomy along unstable
leaves and so, for $\Pi^u_f(\FF^c(x),\FF^c(y))$.

Recall that  a diffeomorphism $f$ in $\EE^r$  has Global Product Structure, that is, there exists a
covering map $\pi:\tM\to M$ such that denoting by $\tFF^*, *=s,u,cs,cu,c$
the lift of the stable, unstable, center stable, center unstable, and
center leaves respectively, then for every $\tx, \ty \in \tM$ we have:
$$
\#\{\tilde{\FF}^{cs}(\tilde{x})\cap\tilde{\FF}^u(\tilde{y})\}=1\,\,\,
\mbox{ and }\,\,\, \#\{\tilde{\FF}^{cu}(\tilde{x})\cap
\tilde{\FF}^s(\tilde{y})\}=1.
$$
This allows to define a continuous map
\begin{equation}\label{eq.proyu}
\Pi^u_\tx:\tM\to\tFF^{cs}(\tx),
\end{equation}
defined by the holonomy map along unstable leaves:
$$
\Pi^u_\tx(\tz)=\tFF^u(\tz)\cap\tFF^{cs}(\tx).
$$

If we restrict the map $\Pi^u_\tx$ to $\tFF^{cu}(\tx),$ we have that
$\Pi^u_\tx(\tFF^{cu}(\tx))=\tFF^c(\tx).$
In an analogous way we define $\Pi^s_\tx$ the holonomy along stable leaves.
And also we define $\Pi^{su}_\tx:\tM\to\tFF^c(\tx)$ by
\begin{equation}
\label{eq.proysu}
\Pi^{su}_\tx=\Pi^s_\tx\circ\Pi^u_\tx.
\end{equation}
Note that, in general, $\Pi^{us}_\tx \neq \Pi^{su}_\tx$.

Recall that we have defined the accessibility class of $x\in M$ as
$$
AC (x) = \{y \in M \colon \text{there is a } su\mbox{-path from } x
\mbox{ to }y\}.
$$
We define the \emph{center accessibility class} of $x$ as
$C(x)=AC(x)\cap \FF^c(x)$.

The same definitions for $\tM:$ for $\tx\in \tM$ its \emph{accessibility
class} is $\widetilde{AC}(\tx)$ = $\{\ty \in \tilde{M} \colon$ there is a
$su$-path from $\tilde{x}$ to $\tilde{y}\}$ and $\widetilde{C}(\tx) =
\widetilde{AC}(\tx) \cap \widetilde{\mathcal{F}}^c(\tx)$.

Let us observe that if
\begin{equation}
\label{eq.proyacc}
\tx, \tz\in\tM \text{ with }\tz \in \widetilde{AC}(\tx)
\Longrightarrow \Pi^{su}_\tx(\tz) \in \widetilde{C}(\tx).
\end{equation}

\begin{lemma}
\label{proj}
Let $\tx\in\tM$ and set $x=\pi(\tx)$.
Then
\begin{itemize}
\item $\pi(\widetilde{AC}(\tilde{x})) = AC (x)$.
\item $\pi(\widetilde{C}(\tilde{x})) \subset C(x)$.
\end{itemize}
\end{lemma}

\begin{proof}
Note that the projection of a $su$-path in $\tilde{M}$ is a $su$-path in
$M$.
Then, $\pi(\widetilde{AC}(\tilde{x})) \subset AC(x)$.
Reciprocally, the lift of a $su$-path in $M$ is a $su$-path in
$\tilde{M}$.
The second part also follows easily:
\begin{align*}
\pi(\widetilde{C}(\tilde{x})) & = \pi(\widetilde{AC}(\tilde{x}) \cap
\tilde{\mathcal{F}}^c(\tilde{x})) \subset \pi(\widetilde{AC}(\tilde{x}))
\cap \pi(\tilde{\mathcal{F}}^c(\tilde{x})) \\
& = AC (x) \cap \mathcal{F}^c(x) = C(x).
\end{align*}
The proof is complete.
\end{proof}

\begin{lemma}
\label{equiv}
The following are equivalent:
\begin{itemize}
\item[a)] $AC (x)$ is an open subset.
\item[b)] $AC (x)$ has non-empty interior.
\item[c)] $C (x)$ is an open subset of $\mathcal{F}^c(x)$
\item[d)] $C (x)$ has non-empty interior (in $\mathcal{F}^c(x)$).
\end{itemize}
\end{lemma}

\begin{proof}
Notice that, by continuity of $\mathcal{F}^s$, if $U$ is an open set in
$M$ then the saturation by stable leaves, i.e., $\cup_{x\in U}\FF^s(x)$
is  also open in $M$ and the same for the saturation by unstable leaves.
Moreover, from the local product structure due the partially hyperbolic
structure, given an open set $V$ in a center leaf, its saturation by
stable and unstable leaf is also an open set in $M$.
From these simple facts the lemma follows.
\end{proof}

The same is true for the lift and also equivalent to the above:
\begin{lemma}
\label{tildeequiv}
The following are equivalent:
\begin{itemize}
\item[a)] $\widetilde{AC} (\tilde{x})$ is an open subset.
\item[b)] $\widetilde{AC} (\tilde{x})$ has non-empty interior.
\item[c)] $\widetilde{C} (\tilde{x})$ is an open subset of
$\tilde{\mathcal{F}}^c(x)$.
\item[d)] $\widetilde{C} (\tilde{x})$ has non-empty interior (in
$\tilde{\mathcal{F}}^c(x)$).
\item[e)] $C (x)$ is an open subset of $\mathcal{F}^c(x)$.
\end{itemize}
\end{lemma}

\begin{lemma}
\label{l.4}
For any $z\in M$ and any center leaf $\FF^c(x)$ we have that $AC(z)\cap
\FF^c(x)\neq\emptyset.$
In particular, $f$ is accessible if and only if for some $x$ it holds that
$C(x) = \mathcal{F}^c(x)$.
\end{lemma}

\begin{proof}
The Global Product Structure implies in particular that for any $z$ and
$x$ we have that:
$$
\FF^u(z)\cap \FF^{cs}(x)\neq \emptyset\;\;\;\mbox{ and }\;\;\;\FF^{cs}(x)
=\bigcup_{y\in\FF^c(x)}\FF^s(y)
$$
which yields $AC(z)\cap \FF^c(x)\neq\emptyset$.
The second part follows immediately.
\end{proof}

Recall that for $y\in\FF^s(x)$ the (local) holonomy map
$\Pi^s(\FF^c(x),\FF^c(y))$ from a neighborhood $U^c_x$ of $x$ in $\FF^c(x)$
to a neighborhood $U^c_y$ of $y$ in $\FF^c(y)$ is a $C^1$ diffeomorphism.
If $\gamma:[0,1]\to \FF^s(x)$ is a path joining $x$ and $y$ then there
exists a continuous  map $\Gamma^s: U^c_x\times [0,1]\to \FF^{cs}(x)$ such
that:
\begin{itemize}
\item $\Gamma^s(z,t) \in \FF^s(z)\;\forall \;t\in[0,1]$.
In particular $\Gamma^s(z,t) \in AC(z)$.
\item $\Gamma^s(z,0)=z$ for all $z\in U^c_x$.
\item $\Gamma^s(z,1)=\Pi^s(\FF^c(x),\FF^c(y))(z)\in U^c_y.$
\item $\Gamma_0^s: U^c_x\to U_y^c$ defined by $\Gamma_0^s(z)=\Gamma(z,1)
=\Pi^s(\FF^c(x),\FF^c(y))(z)$ is a $C^1$ diffeomorphism.
\end{itemize}

The same holds when $y\in\FF^u(x)$ considering $\Pi^u(\FF^c(x),\FF^c(y))$
and also for the lift $\tilde{f}:\tM\to\tM.$
In particular, if $y\in AC(x)$ and $\gamma$ is $su$-path joining $x$ to
$y$, by finite composition of maps as above we get a map (see
Figure~\ref{mapgamma})
\begin{eqnarray}
\label{eq.gamma}
\Gamma \colon U^c_x\times [0,1]\to M
\end{eqnarray}
such that
\begin{itemize}
\item  $\Gamma(z,t) \in AC(z)$ for all $t\in[0,1].$
\item $\Gamma(z,0)=z$ for all $z\in U^c_x.$
\item $\Gamma(z,1)\in U^c_y.$
\item $\Gamma_0 \colon U^c_x\to U^c_y$ defined by $\Gamma_0(z)=\Gamma(z,1)$
is a $C^1$ diffeomorphism.
\end{itemize}

\begin{figure}
\centering
\psfrag{cx}{$\FF^c(x)$}
\psfrag{cy}{$\FF^c(y)$}
\psfrag{s}{$s$}
\psfrag{u}{$u$}
\psfrag{Ux}{$U^c_x$}
\psfrag{Uy}{$U^c_y$}
\includegraphics[height=7cm]{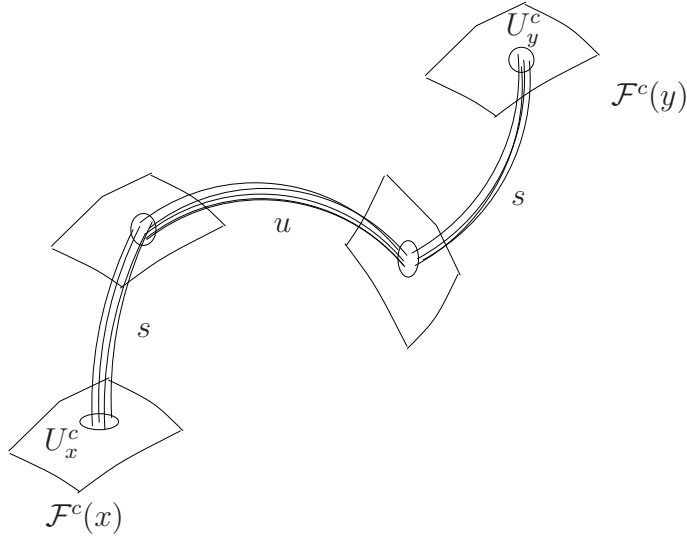}
\caption{The map $\Gamma.$}
\label{mapgamma}
\end{figure}

In the same way, if $\ty\in\widetilde{AC}(\tx)$ we have a map
\begin{equation}\label{eq.mapacc}
\tilde{\Gamma} \colon U^c_{\tx}\times[0,1]\to\tM
\end{equation}
with the properties above.
The above has important consequences.
First, we recall the definition of a homogeneous subset.
\begin{definition}
\label{homogeneous}
Let $X$ be a riemannian manifold.
A subset $Z\subset X$ is said {\em $C^r$-homogeneous} if for every pair of
points $x,y \in Z$ there are neighborhoods $U_x, U_y \subset X$ of $x$ and
$y$, respectively, and a $C^r$-diffeomorphism $\phi \colon U_x \to U_y$
such that $\phi (U_x \cap Z) = U_y \cap Z$ and $\phi(x) = y$.
\end{definition}

Thus, it follows straightforward from the definitions of $\Gamma$ and
$\tilde{\Gamma}$ that:
\begin{lemma}
\label{l.first}
$C(x)$ and $\widetilde{C}(\tx)$ are $C^1$-homogeneous.
\end{lemma}

Moreover, we have the following:
\begin{lemma}
\label{l.path} Let $\tx\in\tM$  and $\ty\in \widetilde{C}(\tx).$
Then there are neighborhoods $U^c_{\tx}$ and $U^c_{\ty}$ in $\tFF^c(x)$
and a continuous map $\tilde{\gamma} \colon U^c_{\tx}\times [0,1]\to
\tFF^c(\tx)$ such that
\begin{enumerate}
\item $\tgamma(\tz,0) = \tz$ for all $\tz \in U^c_{\tx}$;
\item $\tgamma(\tz,1) \in U^c_{\ty}$ for all $\tz \in U^c_{\tx}$;
\item $\tgamma(\tx,1) = \ty$; and
\item $\tgamma(\tz,t) \in \widetilde{C}(\tz)$, for all
$\tz \in U^c_{\tx}$ and $t\in [0,1]$.
\end{enumerate}
\end{lemma}

\begin{proof}
Just take $\tgamma=\Pi^{su}_\tx\circ \tilde{\Gamma}$ where $\Pi^{su}_\tx$
and $\tilde{\Gamma}$ are the maps defined in \eqref{eq.proysu} and
\eqref{eq.mapacc}, respectively.
See also \cite{RH05}).
\end{proof}

\begin{remark}
\label{path2}
Notice that projecting by $\pi \colon \tM\to M$ we have a similar result:
given $x$ in $M$ and $\tx\in\tM$ such that $\pi(\tx)=x$ and $y\in
\pi(\wtC(\tx))$ then there are neighborhoods $U^c_{x}$ and $U^c_{y}$
in $\FF^c(x)$ and a continuous map $\gamma \colon U^c_{x}\times [0,1]\to
\FF^c(x)$ such that
\begin{enumerate}
\item $\gamma(z,0) = z$ for all $z \in U^c_{x}$;
\item $\gamma(z,1) \in U^c_{y}$ for all $z \in U^c_{x}$;
\item $\gamma(x,1) = y$; and
\item $\gamma(z,t) \in C(z)$, for all $z \in U^c_{x}$ and $t\in [0,1]$.
\end{enumerate}
\end{remark}

\begin{corollary}\label{tildeconnected}
$\widetilde{C}(x)$ is connected and arc-connected.
\end{corollary}

Recall that an accessibility class $AC(x)$ is trivial if
$AC(x)\cap\FF^c(x)$ is totally disconnected.
Notice that by the map $\Gamma$ defined in \eqref{eq.gamma} this does not
depend on $x,$ just on the accessibility class.

\begin{lemma}
The accessibility class $AC(x)$ is trivial if and only if
$\widetilde{C}(\tx)=\{\tx\}$ where $\pi(\tx)=x.$
\end{lemma}
\begin{proof}
If $\widetilde{C}(\tx)\neq\{\tx\}$ then, since it is connected and
arc-connected we have that $\pi(\widetilde{C}(\tx))\subset C(x)$ contains
a non trivial connected set and so $C(x)$ is not totally disconnected.
On the other hand, if $\widetilde{C}(\tx)=\{\tx\}$ then
$\widetilde{AC}(\tx) \cap \tFF^c(\ty)$ consists of a single point for any
$\ty$ and in particular for those $\ty$ with $\pi(\ty)=x.$
Therefore $C(x)$ is at most countable and so totally disconnected.
\end{proof}

\begin{corollary}
\label{nontrivial}
The set $\{x\in M: AC(x) \mbox{ is nontrivial}\;\}$ is open in $M.$
\end{corollary}

\begin{proof}
Let $x\in M$ be such that $AC(x)$ is nontrivial and let $\tx$ such that
$\pi(\tx)=x.$
Then $\widetilde{C}(\tx)\neq\{\tx\}.$
Now, from Lemma~\ref{l.path} we get for any $\tz \in \tFF^c(\tx)$ close
enough to $\tx$ that $\widetilde{C}(\tz)\neq\{\tz\}$ and the lemma
follows.
\end{proof}

The above says that if we have a nontrivial accessibility class then
nearby the classes are nontrivial.
Indeed, the same holds when we perturb also $f:$

\begin{corollary}
\label{nontrivunif}
Let $f\in\EE^r$ and let $x\in M$ be such that $AC(x)$ is nontrivial.
Then, there exist a neighborhood $U_x$ of $x$ and a neighborhood
$\UU(f)$ in $\EE^r$ (which can be considered $C^1$ open as well) such that
for any $g\in\UU(f)$ and $z\in U_x$ the accessibility class $AC(z,g)$ is
nontrivial.
\end{corollary}

\begin{proof}
Let $\tx$ be such that $\pi(\tx)=x.$
Since $AC(x)$ is nontrivial, then $\wtC(\tx)\neq\{\tx\}.$
Let $\ty\in\wtC(\tx)$, $\ty\neq\tx.$
We have a $su$-path (in $\tM$) joining $\tx$ to $\ty.$
Now, by continuity of stable and unstable manifolds on compact sets (with
respect to the point and to $f$) there are disjoint open sets $U_\tx$ and
$U_\ty$ of $\tx$ and $\ty$ in $\tM$ and an open set $\UU(f)$ such that
for any $g$ in $\UU(f)$ we have that any point in $U_\tx$ can be joined by
$su$-path of $\tilde{g}$ with a point in $U_\ty.$
And moreover, if we consider $\Pi^{su}_{\tilde{g}}\colon \tM\to
\FF^c(\tx, \tilde{g})$, then $\Pi^{su}_{\tilde{g}}(U_\tx)$ and
$\Pi^{su}_{\tilde{g}}(U_\ty)$ are open sets (in $\FF^c(\tx,\tilde{g})$)
and disjoint.
The result thus follows, since for any $\tz\in \Pi^{su}_{\tilde{g}}
(U_\tx)$ there is point in $\Pi^{su}_{\tilde{g}}(U_\ty)$ that belongs to
$\wtC(\tz,g).$
\end{proof}

\begin{lemma}
\label{disj}
Let $x\in M$ and let $\tx\in\tM$ with $\pi(\tx)=x.$ Let $\tFF^c_1$ and
$\tFF^c_2$ be such that $\pi(\tFF^c_i)=\FF^c(x)$ and let
$\tC_i=\wtAC(\tx)\cap\tFF^c_i,\,i=1,2.$
Then either $\pi(C_1)$ and $\pi(C_2)$ are equal or disjoint.
\end{lemma}

\begin{proof}
Assume that $\pi(\tC_1)\cap\pi(\tC_2)\neq\emptyset.$
Let $z$ be in this intersection and let $\tz_i\in \tC_i$ be such that
$\pi(\tz_i)=z.$
It follows that $\tC_i=\wtC(\tz_i).$ Let $\beta$ be a covering map,
$\beta(\tz_1)=\tz_2.$
Since $\beta$ sends $su$-path in $\tM$ to $su$-path we conclude that
$\beta(\tC_1)\subset \tC_2$ and $\beta^{-1}(\tC_2)\subset \tC_1.$
Hence $\pi(\tC_1)=\pi(\tC_2).$
\end{proof}

\begin{corollary}
Let $x\in M$ and $\tx\in\tM$ such that $\pi(\tx)=x.$
If $C(x)$ is open then $\pi(\wtC(\tx))$ is the connected component of
$C(x)$ that contains $x.$
\end{corollary}

\begin{proof}
It is a direct consequence of Lemmas \ref{proj}, \ref{tildeequiv}, and
\ref{disj}.
\end{proof}

We now investigate the structure of accessibility classes when the center
bundle has dimension two, that is $f\in\EE^r$ and $\dim E^c=2.$
The following important result is essentially contained in \cite{RH05}.
\begin{theorem}
Let $f\in\EE^r$ and assume that $\dim E^c=2.$ Let $\tx\in\tM.$
Then one and only one of the following holds:
\begin{enumerate}
\item $\wtC(\tx)$ is open;
\item $\wtC(\tx) = \{\tx\}$;
\item $\wtC(\tx)$ is a $C^1$ one dimensional manifold without boundary.
\end{enumerate}
\end{theorem}

\begin{proof}
The same proof in \cite[Proposition 5.2]{RH05} yields that $\wtC(\tx)$ is
either open, consists just of $\tx$ or it is a topological one dimensional
manifold.
Now, in case $\wtC(\tx)$ is a topological one dimensional manifold, by the
$C^1$ homogeneity of $\wtC(\tx)$ and the result in \cite{RSS96} which says
that a locally compact and $C^1$ homogeneous subset of a riemannian
manifold is a $C^1$ submanifold, one get that in fact $\wtC(\tx)$ is of
class $C^1$ (and without boundary).
\end{proof}

Let us denote by $C_0(x)$ the arc-connected component of $C(x)$ that
contains $x.$
We remark that when $C(x)$ is open, then $C_0(x)$ is just the connected
component of $C(x)$ that contains $x.$

\begin{corollary}\label{c.connect}
Let $f\in\EE^r$ and assume that $\dim E^c=2.$
Let $x\in M$ and $\tx\in\tM$ with $\pi(\tx)=x.$ Then $\pi(\wtC(\tx)) =
C_0(x).$
\end{corollary}

\begin{proof}
When $\wtC(x)$ is open or trivial then the result follows
immediately. Thus, we just have to check it when $\wtC(x)$ is a one
dimensional submanifold (without boundary).

From Lemma \ref{proj} and Corollary \ref{tildeconnected} we have
$\pi(\wtC(\tx)) \subset C_0(x)$.
On the other hand, let $\beta$ be an arc in $C_0(x)$ starting at $x$ and
assume that it is not contained in $\pi(\wtC(\tx)).$
Let $t_0=\sup\{t \colon \beta([0,t])\subset\pi(\wtC(\tx))\}.$
Let $y=\beta(t_0)$ (that belongs to $\pi(\wtC(\tx))$).
Since $\wtC(\tx)$ has no boundary, we have an arc $\alpha$ inside
$\pi(\wtC(x))$ having $y$ in its interior, say joining $x_1$ with $x_2.$
Now, applying Lemma \ref{l.path} (or Remark \ref{path2}) we have a
continuous map $\gamma \colon U^c_{x_1}\times [0,1]\to \FF^c(x)$ such that
$\gamma(x_1,t)=\alpha(t)$ and that $\gamma(z,t)\in C(z)$ for any $t$ and
$z\in U_{x_1}.$
Then we conclude that there is an open set $U\subset U_{x_1}$ such that
$\gamma(U\times[0,1])\cap\beta\neq\emptyset$ (see Figure
\ref{f.corconnect}).
This implies that $C(x)$ is open, a contradiction.
Therefore, $C_0(x) \subset \pi(\wtC(\tx))$.
The proof is finished.
\end{proof}

\begin{figure}
\centering
\psfrag{x}{$x$}
\psfrag{z}{$z$}
\psfrag{y}{$y$}
\psfrag{U1}{$U_{x_1}$}
\psfrag{U2}{$U_{x_2}$}
\psfrag{x1}{$x_1$}
\psfrag{x2}{$x_2$}
\psfrag{al}{$\alpha$}
\psfrag{be}{$\beta$}
\psfrag{ga}{$\gamma(z,t)$} \psfrag{u}{$u$}
\includegraphics[height=6cm]{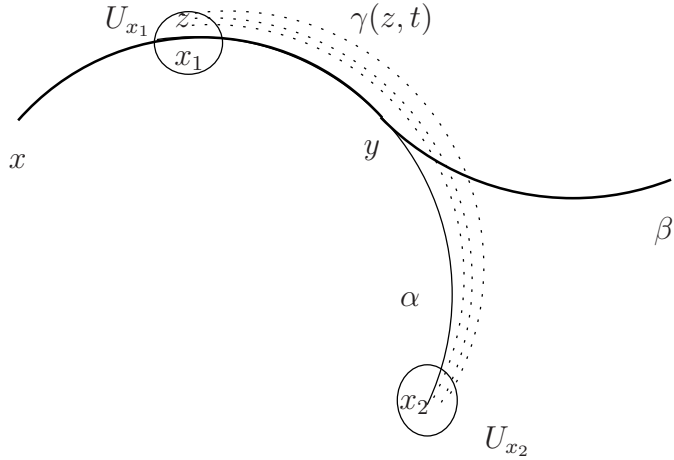}
\caption{Proof of Corollary \ref{c.connect}}
\label{f.corconnect}
\end{figure}

Next, we want to show that the set of one dimensional accessibility
classes form a lamination.
Lets recall the definition of a lamination (we stated for the case of
lamination of subsets of a surface for simplicity.

\begin{definition}
Let $S$ be a surface and let $K\subset S$ be closed.
We say that $K$ admits a $C^1$-lamination if $K$ has a partition into $C^1$
one dimensional manifolds (called leaves of the lamination) such that every
point in $K$ has a neighborhood $U$ homeomorphic to $(-1,1)\times (-1,1)$
(called charts of the lamination) such that $K\cap U$ correspond to $
F\times (-1,1)$, where $F$ is a closed set in $(-1,1)$ and every
$\{\verb"f"\}\times (-1,1)$, for $\verb"f"\in F$, is inside a leaf of $K,$
and tangent spaces of the leaves vary continuously.
\end{definition}

Fix a center leaf $\tilde{\mathcal{F}}^c$.
Let $\tilde{K}$ be the set of accessibility classes in
$\tilde{\mathcal{F}}^c$ which are $C^1$ one-dimensional submanifolds.
Then $\tilde{K}$ is partitioned by the accessibility classes
$\tilde{C}(\tx)$, $\tilde{x} \in \tilde{K}$.
We want to prove that this partition form a $C^1$-lamination.

First we prove that the accessibility classes vary continuously in the
$C^1$-topology:

\begin{proposition}
\label{p.cont-var}
For $\tilde{x} \in \tilde{K}$, the curves $\tilde{C}(\tx)$ vary
continuously in the $C^1$ topology.
\end{proposition}

Before proving the proposition we need an elementary lemma from calculus.
Let us introduce some notations.
Let $\gamma \colon [0,1] \to S$ be an $C^1$-arc of a surface $S$ and
consider an $\varepsilon$-tubular neighborhood $N_\gamma$.
This tubular neighborhood is diffeomorphic to $[0,1] \times [-\varepsilon,
\varepsilon]$.
Given a point in $N_\gamma$ we identify with coordinates $(t,s)$,
$t \in [0,1]$ and $s \in [-\varepsilon, \varepsilon]$.

We call the \emph{left side} of $N_\gamma$ the boundary $\{0\} \times
[-\varepsilon, \varepsilon]$ and \emph{right side} the boundary $\{1\} \
times [-\varepsilon, \varepsilon]$.
We denote by $\pi \colon N_\gamma \to \gamma$ the orthogonal projection,
i.e., in local coordinates $\pi(t,s) = t$.

\begin{lemma}
\label{l.tubular}
With the notations above, given $\delta > 0$, there exists $\varepsilon =
\varepsilon(\delta)$ such that if $\beta$ is a $C^1$-curve in $N_\gamma$
from the left to the right side (and do not intersect $\gamma$) then there
is some $(t,s) = \beta (\tilde{t})$ such that the angle
$\angle (\dot{\beta}(\tilde{t}), \dot{\alpha}(t)) < \delta$.
\end{lemma}

Now we are ready to give a proof of Proposition~\ref{p.cont-var}.
Since we will be working in a neighborhood of $\tFF^c$ we may assume that
we are in trivialization chart of the tangent bundle $T\tFF^c$, and so we
may compare angles and norms of vectors in different tangent spaces.

\begin{proof}[Proof of Proposition~\ref{p.cont-var}]
We need to prove that if $\tx \in \tilde{K}$ and $\tx_n \in \tilde{K}$
converges to $\tx$ then $T_{\tx_n}\tC(\tx_n)$ converges to
$T_{\tx}\tC(\tx)$.
Assume that it is not true.
Then there exists some sequence $\tx_n \in \tK$ converging to some point
$\tx \in \tK$ and some $\eta >0$ such that the angle
$\angle (T_{\tx_n}\tC(\tx_n), T_{\tx}\tC(\tx)) > \eta$, for all $n$.

Let $\ty \in \tC(\tx)$, $\ty \neq \tx$ and let $\Upsilon^{su}$ a $su$-path
joining $\ty$ to $\tx$.
We may consider an arc of $su$-paths, i.e., for each $t \in [0,1]$ a
$su$-path $\Upsilon^{su}_t$ that vary continuously joining $\tx$ with some
point $\Upsilon^{su}_t(1) \in \tC(\tx)$ such that $\Upsilon^{su}_0$ is the
trivial $su$-path and $\Upsilon^{su}_1 = \Upsilon^{su}$.
We may assume that the path $\Upsilon^{su}_t(1)$ is the arc joining $\tx$
and $\ty$ in $\tC(\tx)$ denoted by $[\tx,\ty]$.
We will consider tubular neighborhood $N$ of the arc $[\tx, \ty]$.

The path $\Upsilon^{su}_t$ allows us to consider (see the last item of
properties of the map $\Gamma$ in \eqref{eq.gamma} and equivalent for
\eqref{eq.mapacc}) a map $\phi_t \colon B(\tx,r_t) \to \tFF^c$ which is a
$C^1$ diffeomorphism onto its image that contains $\Upsilon^{su}_t(1)$.
We may choose $r_t = r$ independent of $t$.
The family $\phi_t$ varies continuously in the $C^1$ topology due to the
local holonomy is $C^1$ inside center stable and center unstable leaves,
the center foliation is $C^1$ and the path $\Upsilon^{su}_t$ varies
continuously with $t$.

Given $\theta>0$ there exists $\delta_0>0$ and $\rho_0>0$ such that for any
$t$ if $\dist(\tz,\tx) < \rho_0$ and $\angle(T_\tx\tC(\tx), w) > \theta$
then the angle
\begin{equation}
\label{eq.delta0}
\angle(d(\phi_t)_\tx(T_\tx\tC(\tx)), d(\phi_t)_\tz(w)) =
\angle(T_{\phi_t(\tx)}\tC(\tx)), d(\phi_t)_\tz(w)) > \delta_0.
\end{equation}

On the other hand, given $\delta_1>0$ there exists $\varepsilon_1>0$ such
that if
\begin{equation}
\label{eq.epsilon1}
\dist(\pi(\phi_t(\tz)), \phi_t(\tx)) < \varepsilon_1 \mbox{ then }
\angle(T_{\pi(\phi_t(\tz))}\tC(\tx), T_{\phi_t(\tx)}\tC(\tx)) < \delta_1.
\end{equation}
Notice also that there exists $\rho > 0$ such that for any $t$ if
\begin{equation}
\label{eq.rho}
\dist (\tx,\tz)< \rho \mbox{ then } \dist(\pi(\phi_t(\tz), \phi_t(\tx)) <
\varepsilon_1.
\end{equation}

Consider $\tx_1$ and $\ty_1$ in $\tC(\tx)$ between $\tx$ and $\ty$.
Denote by $\gamma = [\tx_1,\ty_1]$ the arc in $\tC(\tx)$ joining $\tx_1$
and $\ty_1$.

Let $\theta = \eta$ and take $\delta_0 = \delta_0(\theta)$ from
\eqref{eq.delta0}.
Choose $\delta_1>0$ such that $\delta_0 - \delta_1 = \delta > 0$ and let
$\varepsilon_1$ from \eqref{eq.epsilon1}.
Choose $\rho < \rho_0$ such that \eqref{eq.rho} holds.

For this $\delta$ choose an $\varepsilon$ tubular neighborhood
$N_{[\tx_1,\ty_1]}$ as in Lemma~\ref{l.tubular} and such that
$N_{[\tx_1,\ty_1]} \subset \cup_{t\in [0,1]} \phi_t(B(\tx,\rho))$.

Now, if $\tx_n$ is close enough to $\tx$ then $\beta_n(t) = \phi_t(\tx_n)$,
$t\in [0,1]$ is a curve that crosses the $N_{[\tx_1,\ty_1]}$ from the left
to the right side.
On the other hand, if $t$ is such that $\beta_n(t) \in N_{[\tx_1,\ty_1]}$
then $\spann(\dot{\beta}_n)= d(\phi_t(\tx_n))(T_{\tx_n}\tC(\tx_n))$ and so
\begin{align*}
\angle(\dot{\beta}_n(t), T_{\pi(\beta_n(t))} & \tC(\tx)) =
\angle(d\phi_t(\tx_n)T_{\tx_n}\tC(\tx_n),
T_{\pi(\phi_t(\tx_n))}\tC(\tx_n))\\
 & > \angle (d\phi_t(\tx_n) T_{\tx_n}\tC(\tx_n),
d\phi_t(\tx) T_\tx\tC(\tx) - \\
 & \quad  - \angle (T_{\pi(\phi_t(\tx_n))}\tC(\tx_n),
d\phi_t(\tx) T_\tx\tC(\tx)) \\
 & >  \delta_0 - \delta_1 = \delta,
\end{align*}
which is a contradiction with Lemma~\ref{l.tubular}.
\end{proof}

\begin{corollary}
\label{c.lamination}
Let $\tFF^c$ be a center leaf in $\tM$ and assume that there is no trivial
accessibility classes.
Then the set $\tK$ of non-open accessibility classes admits a
$C^1$-lamination whose leaves are accessibility classes.
The same holds for the set of non open accessibility classes $K$ in
$\mathcal{F}^c$ whose leaves are connected components of accessibility
classes $C_0(x)$ for $x\in K$.
\end{corollary}
\begin{proof}
Since there is no trivial accessibility class the set $\tK$ is closed.
From Proposition \ref{p.cont-var} and using transversal sections it is not
difficult to construct a chart for each $\tx\in\tK.$
\end{proof}

\section{Perturbation Lemmas}
\label{s.perturbation}

In this section we prove our main perturbation techniques that allow us to
prove our theorems.
These are Lemmas \ref{l.primeiro} and \ref{l.segundo}.
Before we state and prove these lemmas we need some elementary results,
the first one says that some perturbations of the identity can be thought
as translations in terms of local coordinates, no matter if we are in the
conservative world, symplectic world, etc.

Let $M$ be a manifold of dimension $d$ and let $S \subset M$ be an
embedded submanifold of $M$ of dimension $k$.
Let $z\in S$ and let $U$ be a neighborhood of $z$ in $M$ and let $V$ be
the connected component of $U\cap S$ containing $z$.
We say that we have \emph{local canonical coordinates} in $V$ if we have a
coordinate chart (or parametrization) $\varphi\colon U \to \mathbb{R}^{d}$
and $\varphi(V)=V_0\subset\mathbb{R}^k$ with $\varphi(z) = 0$.

It is a consequence of Darboux Theorem that if $\omega$ is a symplectic
form in $M$ such that $\omega_{/S}$ is symplectic and $k=2j, d=2l$ and we
write coordinates in $\mathbb{R}^k$ as $(x_1, \dots ,x_j,y_1, \dots ,y_j)$
and in $\mathbb{R}^d$ as $(x_1,\dots ,x_l,y_1,\dots ,y_l)$ then we may
assume that the local chart verifies $\varphi^*(\sum_{i=1}^d dx_i\wedge dy_i)
= \omega|U$ and $\varphi^* (\sum_{i=1}^j dx_i\wedge dy_i) = \omega|V$, see
e.g. \cite{Mo65} and \cite{We71}.

And in case $m$ is a volume form it is well known that we can choose local
coordinates in $\mathbb{R}^d$ as $(x_1,\dots ,x_d)$, we may assume also
that $\varphi^*(dx_1\wedge \dots\wedge dx_d)=m,$ i.e., in local coordinates
the volume form is the standard volume form in $\mathbb{R}^d$ (see e.g.
\cite{Mo65}).

The next lemma says that we can glue an arbitrary small translation near a
point with the identity outside a neighborhood in the conservative and
symplectic setting.
Sophisticated versions of this problem can be found in \cite{DM90} and
\cite{AM07} (pasting lemma).

\begin{lemma}
\label{l.pert_sympl}
Let $M$ be a $C^r$ manifold and let $S \subset M$ be a $C^r$ submanifold
of $M, r\ge 1$.
Let $z\in S$ and let $U$ be a given neighborhood of $z$ in $M$ such that
$V := S \cap U$ has local canonical coordinates.
Then, there exist $V'$, $z \in V' \subset V \subset U$ and
$\varepsilon_0 > 0$ such that for any $0 \le \varepsilon \le \varepsilon_0$
there exists $\delta > 0$ such that for any $w \in \mathbb{R}^k$, $\| w \|
< \delta$ there exists a diffeomorphism $h\colon M \to M$  satisfying\,:
\begin{enumerate}
\item $h$ is $\varepsilon$-$C^r$ close to identity;
\item $h \equiv id$ on $U^c$;
\item $h$ preserves $V$, i.e. $h(V) = V$;
\item $h|V'$ in local coordinates is given by
$$
y \mapsto h(y) = y + w.
$$
\end{enumerate}

If $\omega$ is a symplectic form in $M$ and $\omega_{/S}$ is symplectic then
$h$ can be taken to be a  symplectomorphism. If $m$ is a volume form in
$M$, we can take $h$ to preserves the volume form $m.$
\end{lemma}

\begin{proof}
In the general case (i.e neither conservative nor symplectic) the solution
is easy, just take the time $t$ map (with $t$ small enough) of the flow
generated by a vector field $X$ (in the local coordinates) such that
$X(x)=v$ for $v\in\mathbb{R}^k, \|v\|=1$ and it is identically zero
outside a neighborhood of the origin.
The same idea works in the conservative setting taking $X$ to be
divergence free vector field and in the symplectic setting taking a
Hamiltonian vector field.

Lets consider the conservative setting. Recall that we have local
coordinates, i.e., a chart $\varphi\colon U \to \mathbb{R}^{d}$ and
$\varphi(V)=V_0\subset\mathbb{R}^k$ with $\varphi(z) = 0$ and
$\varphi^*(dx_1\wedge \dots\wedge dx_d)=m.$ Let
$\psi:\mathbb{R}^d\to\mathbb{R}$ be a bump function such that it is equal
to $1$ in a neighborhood of $0$ and it is identically zero outside a
neighborhood of $0$ as well (with closure contained in $\varphi(U)$).
Let $w\in\mathbb{R}^k, \|w\|=1.$
By a linear change of coordinates (preserving $\mathbb{R}^k$) we may
assume that $w=e_1.$
Consider the function $\chi:\mathbb{R}^d\to \mathbb{R}$ given by
$$
\chi(x)=\chi(x_1,x_2,\ldots,x_d)=\psi(x)x_2.
$$
Then, taking $X(x)=(\frac{\partial \chi}{\partial x_2},
-\frac{\partial \chi}{\partial x_1}, 0,\ldots, 0)$ we have that $X$ is
divergence free.
Now, taking the time $t$ map of the flow generated by $X$ for $t$ small we
get the lemma.

Lets consider the symplectic case.
We may assume without loss of  generality that $U$ is contained in a
tubular neighborhood of $S$ and $U = U' \times D$, $D$ fibers of the
tubular neighborhood.
Choose open balls $V_1 \subset V_1' \subset U_1 \subset U_1' \subset U$
centered in $z$ and a $C^\infty$-bump function $\psi \colon M \to
\mathbb{R}$ so that $\psi|U^c_1 \equiv 0$ and $\psi|V_1 \equiv 1$.

For simplicity we will assume that $S$ is bidimensional.
Let $u = (u_1,u_2)$ be a unit vector in $\mathbb{R}^2$ and let
$H^u_0 \colon \mathbb{R}^2 \to \mathbb{R}$ be defined by $H^u_0 (x,y) =
yu_1 - xu_2$.
Notice that $X_{H^2_0} = u$ in $\mathbb{R}^2$.

Let $\tilde{H}^u_0 \colon \mathbb{R}^2 \to \mathbb{R}$ be $\tilde{H}^u_0 =
\psi \cdot H^u_0$ and let $H^u_1 \colon U \to \mathbb{R}$ be
$H^u_1 (y) = H^u_0(\varphi(\pi(y)))$, where $\pi$ is the projection along
the fibers of the tubular neighborhood, and let $H^u \colon M \to
\mathbb{R}$ be such that
\begin{itemize}
\item[$\cdot$] $H^u \equiv 0$ in $U^c$, and
\item[$\cdot$] $H^u = \psi(y) H^u_1(y)$ if $y \in U$.
\end{itemize}
Notice that $H^u$ is $C^\infty$ and the $C^r$-norm is bounded by a
constant $K$ that does not depend on $u$.

Let $y \in S \cap V_1$.
We claim that $X_{H^u}(y) = X^S_{H^u}(y)$, where $X^S_H$ is the hamiltonian
field of $H|S$.
Indeed, $X_{H^u}(y)$ is defined as $\omega(X_{H^u}(y), \cdot) = -dH^u_y$
and $X^S_{H^u}(y)$ as $\omega_0(X^S_{H^u}(y), \cdot) = -d(H^u|S)_y$.
For $y\in V_1 \cap S$, $H^u(y) = H^u(\pi(y))$ and so
$dH^u = dH^u|S \circ d\pi$ and hence $dH^u_y|(T_yS)^{\perp_\omega} = 0$.
Thus, for any $v\in (T_yS)^{\perp_\omega} = 0$, then $X_{H^u}(y) \in T_yS$
and so, since for any $w \in T_yS$, we have
$$
\omega(X_{H^u}(y), w) = \omega_0(X^S_{H^u}(y), w).
$$
We conclude that
$$
X_{H^u}(y) = X^S_{H^u}(y).
$$
This proves the claim.
Finally, taking the time $t$ map of the corresponding hamiltonian flow,
for $t$ small enough, we conclude the lemma.
\end{proof}

For $f\in \EE^r$ we denote the {\em stable manifold of size $\varepsilon$
of a center leaf} $\displaystyle \FF^c_1$ by $W^s_\varepsilon(\FF^c_1) :=
\bigcup_{z \in \FF^c_1}(W^s_\varepsilon(z))$.

\begin{remark}
If $\FF^c_1$ is a compact periodic center leaf and $w \in W^s_\varepsilon
(\FF^c_1)$ then there exists $\varepsilon_0$ such that $w \in
W^s_{\varepsilon_0}(\FF^c_1)$ but $w\notin
\overline{f^n(W^s_{\varepsilon_0}(\FF^c_1))}$, for all $n\ge 1$.
\end{remark}

\begin{lemma}
\label{l.prod_estr} Let $f \in \EE^r$ and let $\FF_1^c$ be a periodic
center leaf of $f$. Let $x \in \FF_1^c$ and let $B$ be a small
neighborhood of $x$ in $M$.
Then there exist $p_1, p_2, w_1, w_2, z_1, z_2 \in B$, $\varepsilon_0,
\varepsilon_1, \varepsilon_2 >0$, and $U_1, U_2$ disjoint neighborhoods of
$w_1, w_2$ in $M$ such that, for $i=1,2$,
\begin{enumerate}
\item $\FF^c(p_i)$, $i=1,2$, are periodic compact center leaves;
\item $w_i \in W^s_{\varepsilon_0}(x) \cap W^u_{\varepsilon_i}(p_i)$;
\item $z_i \in W^u_{\varepsilon_0}(\FF^c_1) \cap W^s_{\varepsilon_i}(p_i)$;
\item $U_i \cap f^n(W^s_{\varepsilon_0}(\FF^c_1)) = \emptyset$, for all
$n\ge 1$;
\item $U_i \cap f^{-n}(W^u_{\varepsilon_i}(\FF^c(p_i))) = \emptyset$, for
all $n\ge 1$;
\item $U_i \cap f^{n}(W^s_{\varepsilon_i}(\FF^c(p_j))) = \emptyset$, for
all $n\ge 0$ and $i,j=1,2$; and
\item $U_i \cap f^{-n}(W^u_{\varepsilon_0}(\FF^c_1)) = \emptyset$, for all
$n\ge 0$.
\end{enumerate}
\end{lemma}

\begin{proof}
Let $B_0$, $x\in B_0 \subset B$ be a foliated chart of the center
foliation\;:
$$
\varphi \colon B_0 \to D^{m-k} \times D^k ,
$$
where $D^{m-k}$ is a disk in $\mathbb{R}^{m-k}$ and $D^k$ is a disk in
$\mathbb{R}^k$, and the center foliation in $B_0$ through $\varphi$ is
$\{y\} \times D^k$ and $\varphi(x) = (0,0)$.

Let $P_x$ be the plaque of $x$ in $B_0$.
We may assume that $\FF^c_1 \cap B_0 = P_x$. In $B_0$, we identify point
in the same plaque, $B_0/_\thicksim \cong D^{m-k}$. Let $P\colon B_0 \to
B_0/_\thicksim$ the projection map.

For $z \in B_0$, denote by $\hat{z} := P(z)$ the plaque of $z$ and
denote by $W^{c*}_{B_0}(z)$, $*=s,u$, the connected component of
$W^*_\varepsilon (F^c(z)) \cap B_0$ that contains $z$, and by
$W^*_{B_0}(\hat{z}) := P(W^{c*}(z))$ ($B_0$ is small compared to
$\varepsilon$).

In a neighborhood $W$ of $\hat{x}$ we have local product structure.
Since periodic compact center leaves are dense, we may choose
$\hat{p}_1 \in W$ such that $\hat{p}_1$ is contained in a compact
periodic center leaf. (This compact center leaf may intersects
$B_0$ in other plaques than $\hat{p}_1$, but if it does, intersects
finitely many plaques in $B_0$).

Let $\hat{w}_1 = W^s_{B_0}(\hat{x}) \cap W^u_{B_0} (\hat{p}_1)$ and
$w_1 = W^s_\varepsilon (x) \cap P^{-1}(\hat{w}_1)$, $p_1 =
W^u_\varepsilon (w_1) \cap P^{-1}(\hat{p}_1)$, $\hat{z}_1 =
W^u_{B_0}(\hat{x}) \cap W^s_{B_0}(\hat{p}_1)$, and $z_1 =
W^s_\varepsilon (p_1) \cap P^{-1}(\hat{z}_1)$.

\begin{figure}[phtb]
\centering
\psfrag{x1}{$\hat{x}$}
\psfrag{x}{$x$}
\psfrag{p1}{$\hat{p}_1$}
\psfrag{p}{$p_1$}
\psfrag{z1}{$\hat{z}_1$}
\psfrag{z}{$z_1$}
\psfrag{w1}{$\hat{w}_1$}
\psfrag{w}{$w_1$}
\psfrag{F}{$\FF_1^c$}
\psfrag{s}{$s$} \psfrag{u}{$u$}
\includegraphics[height=5cm]{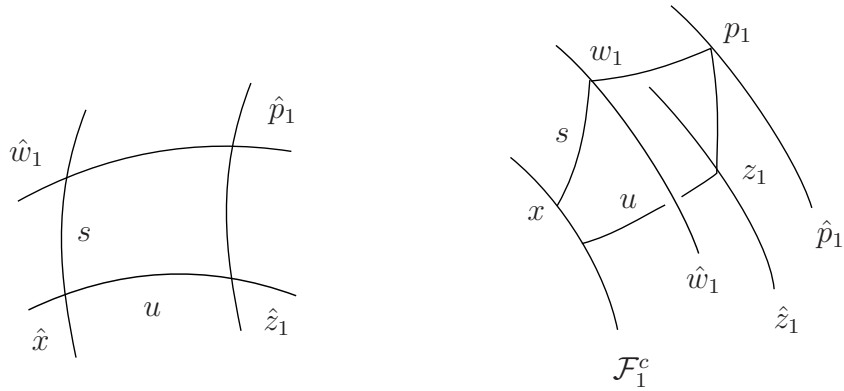}
\caption{$4$-legged path}
\label{f.4-legged}
\end{figure}

Now, we can find $\varepsilon_0, \varepsilon_1 > 0$ such that
\begin{itemize}
\item[$\cdot$] $w_1 \in W^s_{\varepsilon_0}(\FF^c_1)$ but $w_1 \notin
\overline{f^n(W^s_{\varepsilon_0}(\FF^c_1))}$, $n \ge 1$, and
\item[$\cdot$] $w_1 \in W^u_{\varepsilon_1}(\FF^c(p_1))$ but $w_1 \notin
\overline{f^{-n}(W^u_{\varepsilon_1}(\FF^c(p_1)))}$, $n \ge 1$.
\end{itemize}

Now, we may take $\hat{p}_2$ close to $\hat{p}_1$ and contained in a
compact periodic leaf and such that, if we set $\hat{w}_2 =
W^s_{B_0}(\hat{x}) \cap W^u_{B_0} (\hat{p}_2)$, $w_2 =
W^s_\varepsilon (x) \cap P^{-1}(\hat{w}_2)$, $p_2 =  W^u_\varepsilon
(w_2) \cap P^{-1}(\hat{p}_2)$, $\hat{z}_2 = W^u_{B_0}(\hat{x}) \cap
W^s_{B_0}(\hat{p}_2)$, and $z_2 = W^s_\varepsilon (p_2) \cap
P^{-1}(\hat{z}_2)$ then, $w_2 \in W^s_{\varepsilon_0}(\FF^c_1)$ but
$w_2 \notin \overline{f^n(W^s_{\varepsilon_0}(\FF^c_1))}$, for
$n \ge 1$, if $\hat{p}_2$ is close enough to $\hat{p}_1$.
Now, we may choose $\varepsilon_2 >0$ such that $w_2 \in
W^u_{\varepsilon_2}(\FF^c(p_2))$ but $w_2 \notin
\overline{f^{-n}(W^u_{\varepsilon_2}(\FF^c(p_1)))}$, $n \ge 1$.

\begin{figure}[phtb]
\centering
\psfrag{x}{$x$}
\psfrag{p1}{$p_1$}
\psfrag{p2}{$p_2$}
\psfrag{z1}{$z_1$}
\psfrag{z2}{$z_2$}
\psfrag{w1}{$w_1$}
\psfrag{w2}{$w_2$}
\psfrag{F}{$\FF_1^c$}
\psfrag{s}{$s$}
\psfrag{u}{$u$}
\includegraphics[height=5cm]{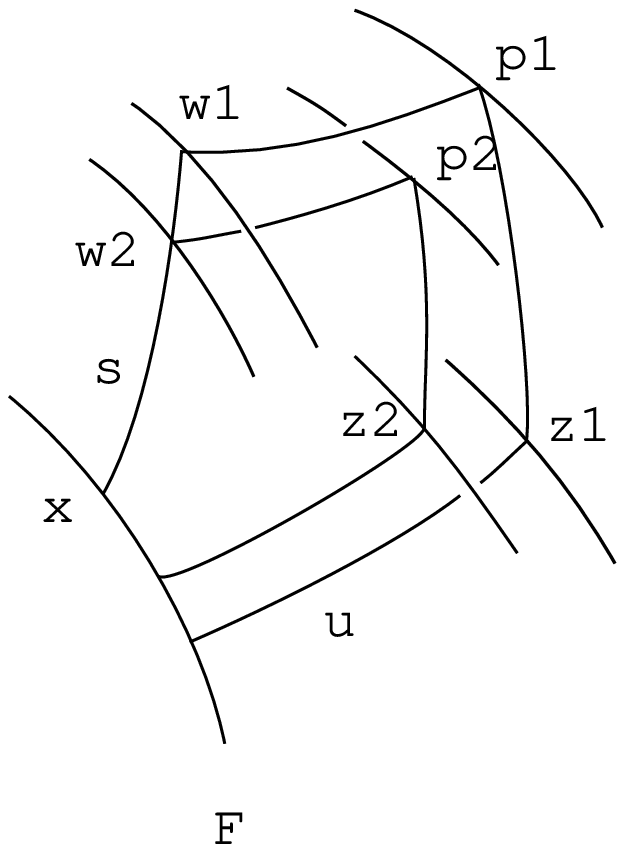}
\caption{Lemma~\ref{l.prod_estr}}
\label{f.lemma}
\end{figure}

Observe that trivially it holds that,
\begin{align*}
& w_i \notin \overline{f^n(W^s_{\varepsilon}(\FF^c(p_j)))},
\quad n \ge 0, \quad i,j = 1,2; \text{ and} \\
& w_i \notin \overline{f^{-n}(W^u_{\varepsilon}(\FF^c_1))}, \quad n
\ge 0, \quad i=1,2.
\end{align*}
From this it is easy to find the neighborhoods $U_1,U_2$. The proof
is complete.
\end{proof}

Given $f\in \EE^r$ and $w \in W^s_\varepsilon(x)$ recall that  we denote
$\Pi^s_f(\FF^c(x), \FF^c(w))$ the (local) holonomy map along the stable
foliation in $\FF^{cs}(x)$ from a neighborhood of $x$ in $\FF^c(x)$ onto
a neighborhood of $w$ in $\FF^c(w)$.

\begin{remark}
\label{equal}
In the setting of Lemma~\ref{l.prod_estr} notice that
$\Pi_f^u(\FF^c(z_i), \FF^c(x))(z_i)$ belongs to $C_0(x)$.
Moreover, if $h \colon M \to M$ is a diffeomorphism close to the identity
such that $h \equiv id$ in $(U_1 \cup U_2)^c$ and preserves $V_i$ the
connected component of $U_i \cup \FF^c(w_i)$ that contains $w_i$ then if
we define $g = f\circ h^{-1}$ we have:
\begin{align}
&\nonumber
\left.
\begin{aligned}
&\FF^c(*,f)=\FF^c(*,g) \quad\mbox{where}\quad *=x,p_1,p_2;\\
&V_i\subset \FF^c(w_i,g);\\
&\FF^c_{loc}(z_i,f)=\FF^c_{loc}(z_i,g);
\end{aligned}
\right. \\
&\left.
\begin{aligned}
& \Pi_g^s(\FF^c(w_i), \FF^c(x)) = \Pi_f^s(\FF^c(w_i), \FF^c(x)) \circ h ;\\
& \Pi_g^u(\FF^c(w_i), \FF^c(p_i)) = \Pi_f^u(\FF^c(w_i), \FF^c(p_i));\\
& \Pi_g^s(\FF^c(z_i), \FF^c(p_i)) = \Pi_f^s(\FF^c(z_i), \FF^c(p_i));
\text{ and}\\
& \Pi_g^u(\FF^c(x), \FF^c(z_i)) = \Pi_f^u(\FF^c(x), \FF^c(z_i)).
\end{aligned}
\right\}
\end{align}
\end{remark}

Denote $\supp (f \neq g) = \{y \colon f(y) \neq g(y)\}$.
The next lemma says that we can destroy trivial accessibility class.

\begin{lemma}
\label{l.primeiro}
Let $\EE$ be as in Theorem~\ref{trivialaccesibility} and let $f\in \EE$
and let $\FF^c(x)$ be a periodic compact center leaf of $f$.
Assume that $C_0(x) = \{ x \}$.
Then, there exists neighborhood $V_x$ of $x$ in $\FF^c(x)$ and
$\varepsilon_0>0$ such that for any $0 < \varepsilon \le
\varepsilon_0$ there exists $g \in \EE$, with $\dist_{C^r} (f,g) <
\varepsilon$ such that:
\begin{enumerate}
\item $\supp (f\neq g)$ is disjoint from the $f$-orbit of $\FF^c(x)$,
\item for any $y \in V_x$ we have $C_0(y,g)) \neq \{ y \}$.
\end{enumerate}
\end{lemma}

\begin{proof}
Let $x,w_1, w_2, p_1, p_2, z_1, z_2$ and $U_1, U_2$ be as in
Lemma~\ref{l.prod_estr}. For simplicity we write,
$\hat{w}_i = \FF^c(w_i,f)$, $\hat{z}_i = \FF^c(z_i,f)$, and
$\hat{p}_i = \FF^c(p_i,f)$.

\begin{figure}[phtb]
\centering
\psfrag{x}{$\hat{x}$}
\psfrag{p1}{$\hat{p}_1$}
\psfrag{p2}{$\hat{p}_2$}
\psfrag{z1}{$\hat{z}_1$}
\psfrag{z2}{$\hat{z}_2$}
\psfrag{w1}{$\hat{w}_1$}
\psfrag{w2}{$\hat{w}_2$}
\includegraphics[height=5cm]{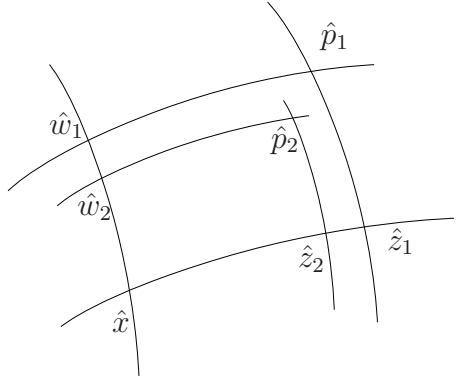}
\caption{Perturbing trivial accessibility classes}
\label{f.perturbation_trivial}
\end{figure}

Since $C_0(x,f) = \{x\}$, we have $\Pi_f^u(\hat{z}_i,\hat{x})(z_i) = x$
and
$$
\Pi^s_f(\hat{w}_i,\hat{x}) \circ \Pi^u_f(\hat{p}_i,\hat{w}_i) \circ
\Pi^s_f(\hat{z}_i,\hat{p}_i) \circ \Pi^u_f(\hat{x},\hat{z}_i) (x) = x,
\quad i=1,2.
$$

Let $W_0$ be a small neighborhood of $x$ in $\FF^c(x)$ and let $W_1
\subset W_0$ such that $W_1 \subset \Pi^s_f(\hat{w}_i,\hat{x}) \circ
\Pi^u_f(\hat{p}_i,\hat{w}_i) \circ \Pi^s_f(\hat{z}_i,\hat{p}_i)
\circ \Pi^u_f(\hat{x},\hat{z}_i) (W_0)$.

We may assume that $U_i$ are small so that if $V_i$ is the connected
component of $U_i\cap\hat{w}_i$ that contains $w_i$ then
$\Pi_f^s(\hat{w}_i,\hat{x})(V_i) \subset W_1$.

Let $\varepsilon >0$ be given (small) and let $V_i'$ as in
Lemma~\ref{l.pert_sympl} for the submanifolds   $\hat{w}_i$
corresponding to $V_i$ and let $V_x \subset
\Pi_f^s(\hat{w}_i,\hat{x})(V_i')$.

Let $l_i \colon V_i \to \hat{w}_i$ be defined by
$$
l_i = \Pi^u_f(\hat{p}_i,\hat{w}_i) \circ
\Pi^s_f(\hat{z}_i,\hat{p}_i) \circ \Pi^u_f(\hat{x},\hat{z}_i) \circ
\Pi^s_f(\hat{w}_i,\hat{x}).
$$
Note that $l_i$ is a $C^1$ map and
$$
\Pi^s_f(\hat{w}_i,\hat{x}) \circ \Pi^u_f(\hat{p}_i,\hat{w}_i) \circ
\Pi^s_f(\hat{z}_i,\hat{p}_i) \circ \Pi^u_f(\hat{x},\hat{z}_i) =
\Pi^s_f(\hat{w}_i,\hat{x}) \circ l_i \circ
(\Pi^s_f(\hat{w}_i,\hat{x}))^{-1}.
$$
Lets look $l_1$ in the local canonical coordinates and let $v$ in
$\mathbb{R}^k$, $\|v\| < \delta$ be such that $-v$ is a regular value of
$l_1-id$.
For this $v$, choose $h_1 \colon M\to M $ as in Lemma~\ref{l.pert_sympl}
(in the appropriate setting for $\EE$) and so $h_1 \circ l_1$ has finitely
many fixed points in $\overline{V'_1}$.
Indeed, if $q$ is a fixed point in $\overline{V'_1}$ of $h_1\circ l_1$
then $l_1(q)-q=-v$ and since $-v$ is a regular value of $l_1-id$ it is an
isolated fixed point.

Let $q_1, \dots , q_\ell$ be the projection of these fixed points in
$V_2$, i.e, $\{ q_1, \dots , q_\ell \}$ $= \Pi^s_f(\hat{w}_1, \hat{w}_2)
(\Fix(h_1 \circ l_1|\overline{V}_1))\cap V_2$.
Choose $h_2 \colon M  \to M $ as in Lemma~\ref{l.pert_sympl} (corresponding
to $U_2, V_2)$ such that no $q_i$ is fixed by $h_2 \circ l_2$.

Let $g \colon M \to M $ be $g=f\circ h^{-1}$ where $h \colon M \to M $ is
defined by
$$
h(z) = \left\{ \begin{array}{cl}
h_1(z) & \text{if } z \in U_1, \\
h_2(z) & \text{if } z \in U_2, \\
z & \text{otherwise}
\end{array}
\right.
$$
It is not difficult to see that (see Remark \ref{equal}):
$$
\Pi^s_g(\hat{w}_i,\hat{x}) \circ \Pi^u_g(\hat{p}_i,\hat{w}_i) \circ
\Pi^s_g(\hat{z}_i,\hat{p}_i) \circ \Pi^u_g(\hat{x},\hat{z}_i) =
\Pi^s_f(\hat{w}_i,\hat{x}) \circ h_i \circ l_i \circ
(\Pi^s_f(\hat{w}_i,\hat{x}))^{-1}.
$$

Now, the maps
$$
\Pi^s_g(\hat{w}_1,\hat{x}) \circ \Pi^u_g(\hat{p}_1,\hat{w}_1) \circ
\Pi^s_g(\hat{z}_1,\hat{p}_1) \circ \Pi^u_g(\hat{x},\hat{z}_1)
$$
and
$$
\Pi^s_g(\hat{w}_2,\hat{x}) \circ \Pi^u_g(\hat{p}_2,\hat{w}_2) \circ
\Pi^s_g(\hat{z}_2,\hat{p}_2) \circ \Pi^u_g(\hat{x},\hat{z}_2)
$$
have no common fixed point.
Thus, for $y \in V_ x$ we have that either for $i= 1$ or $2$ that
$$
\Pi^s_g(\hat{w}_i,\hat{x}) \circ \Pi^u_g(\hat{p}_i,\hat{w}_i) \circ
\Pi^s_g(\hat{z}_i,\hat{p}_i) \circ \Pi^u_g(\hat{x},\hat{z}_i) (y) \neq y.
$$
This completes the proof.
\end{proof}

We need the following elementary result.
For completeness, we give a proof in the appendix.
It says roughly that two nondecreasing maps of the interval with
arbitrarily small translations have no fixed points in common (this is
very simple when the maps are $C^1$ by transversality).

\begin{proposition}
\label{p.bb}
Let $\ell_1 \colon [-a,a] \to \mathbb{R}$ and $\ell_2 \colon [-b,b] \to
\mathbb{R}$ be two non-decreasing maps and let $\phi \colon [-b,b] \to
[-a,a]$ be also a non-decreasing map.
Then for any $\varepsilon > 0$ there exist $s,t$, $|s|,|t|\le \varepsilon$,
such that\,:
$$
\phi(\{ x\in [-b,b] \colon \ell_2(x)+t = x \} ) \cap \{ x\in [-a,a]
\colon \ell_1(x)+s = x \} = \emptyset.
$$
\end{proposition}

\begin{proof}
See Appendix~\ref{a.apA}.
\end{proof}

We now present the last lemma of this section and it will play a key role
in the proof of our main result (Theorem \ref{main}).

\begin{lemma}
\label{l.segundo} Let $\EE$ be either $\EE^r_\omega$ or $\EE^r_{sp,\omega}$
(i.e. as in Theorem \ref{main}).
Consider $f \in \EE$ and let $\FF^c_1$ be a periodic compact center leaf.
Let $x \in \FF_1^c$ and assume that $C_0(x)$ is a $C^1$-simple closed
curve $\CC$.
Let $U$ be a neighborhood of $\CC$ homeomorphic to an annulus and assume
that a family $\Gamma$ of disjoint essential simple closed curves
contained in $U$ is given, with $\CC \in \Gamma$.
Then, there exist a neighborhood $V$ of $\CC$ homeomorphic to an annulus
and $\varepsilon_0 > 0$ such that for any $0 < \varepsilon \le
\varepsilon_0$ there exist $g \in \EE$ such that
\begin{enumerate}
\item $\dist_{C^r}(f,g) < \varepsilon$,
\item $\supp (f \neq g)$ is disjoint from the $f$-orbit of $\FF_1^c$, and
\item no curve of $\Gamma$ contained in $V$ is the accessibility class
of a point in $V$, i.e, for any $y\in V,\; C_0(y,g)\neq \gamma$ for any
$\gamma\in\Gamma.$
\end{enumerate}
\end{lemma}

\begin{proof}
Let $x$ be as in statement of the lemma and let $x,w_1, w_2, p_1, p_2, z_1,
z_2$ and $U_1, U_2$ be as in Lemma~\ref{l.prod_estr}.
Again, for simplicity we denote, $\hat{w}_i = \FF^c(w_i,f)$, $\hat{z}_i =
\FF^c(z_i,f)$, and $\hat{p}_i = \FF^c(p_i,f)$.
Let $V_i$ be the connected component of $U_i\cap \hat{w_i}$ that contains
$w_i$ and let $V_i'$ be as in Lemma~\ref{l.pert_sympl}.
Let $W \subset \Pi^s_f (\hat{w}_i,\hat{x})(V_i')$ be open and containing
$x$, $i=1,2$, and let $\CC_w = W \cap \CC$ (we may assume that $\CC_w$ is
an arc). Let $\CC_i = \Pi^s_f(\hat{x},\hat{w}_i)(\CC_w)$, $i=1,2$ (see
Figure~\ref{f.perturbation_curves}).

\begin{figure}[phtb]
\centering
\psfrag{x}{$\hat{x}$}
\psfrag{p1}{$\hat{p}_1$}
\psfrag{p2}{$\hat{p}_2$}
\psfrag{z1}{$\hat{z}_1$}
\psfrag{z2}{$\hat{z}_2$}
\psfrag{w1}{$\hat{w}_1$}
\psfrag{w2}{$\hat{w}_2$}
\psfrag{U}{$U$}
\psfrag{W}{$W$}
\psfrag{C}{$\CC$}
\psfrag{S1}{$S_1$}
\psfrag{S2}{$S_2$}
\psfrag{I1}{$I_1$}
\psfrag{I2}{$I_2$}
\psfrag{C1}{$\CC_1$}
\psfrag{C2}{$\CC_2$}
\psfrag{J1}{$J_2$}
\psfrag{J2}{$J_1$}
\includegraphics[height=7.5cm]{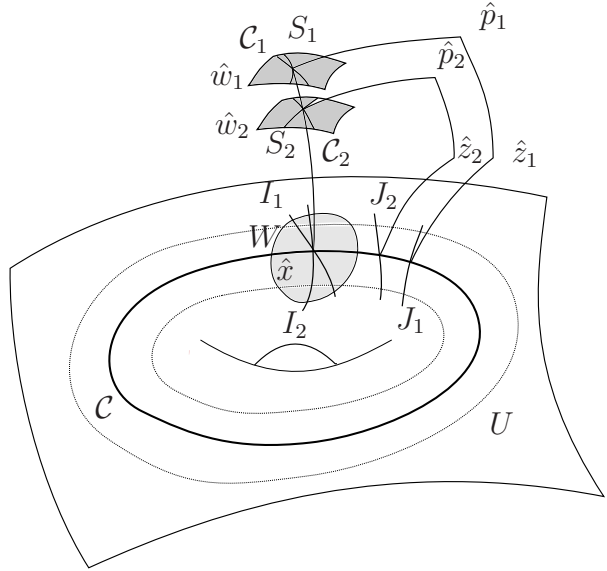}
\caption{Perturbing closed $1$-dimensional accessibility classes.}
\label{f.perturbation_curves}
\end{figure}

In the local canonical coordinates in $V_i$, let $S_i$ be straight segments
transversal to $\CC_i$ at $w_i$, and let $I_i = \Pi^s_f(\hat{w}_i,
\hat{x}) (S_i\cap V_i')$.
These arcs are transversal to $\CC$ at $x$. We take $V$ a compact
neighborhood of $\CC$ homeomorphic to a closed annulus, such that both
$I_i$ crosses $V$ and intersects $\CC$ in just one point.
We may suppose that if $\gamma \in \Gamma$, $\gamma \cap V \neq \emptyset$
then $\gamma \subset V$.
Moreover, we redefine $I_i$ to be the connected component of $I_i\cap V$
that contains $x$ and let $S_i'=\Pi^s_f(\hat{x},\hat{w_i})(I_i) \subset
S_i\cap V_i'.$

Let $\hat{J}_i = \Pi^u_f(\hat{z}_i,\hat{x}) \circ \Pi^s_f(\hat{p}_i,
\hat{z}_i) \circ \Pi^u_f(\hat{w}_i,\hat{p}_i)(S_i)$ and we may also assume
that $\hat{J}_i$ crosses $V$.
Notice that $\hat{J}_i$ are transversal to $\CC$. Let $J_i$ be a connected
component of $\hat{J}_i\cap V$ that crosses $V$ (and we may assume that
$J_i$ intersects $\CC$ in just one point).
We will define maps $P_i \colon I_i \to J_i$, $\varphi \colon I_1 \to I_2$,
and $\psi\colon J_1 \to J_2$ as follows. We will just define $P_1 \colon
I_1 \to J_1$, the others are completely similar.

We order the arcs $I_1, I_2, J_1, J_2$ so that all of them crosses
$\CC$ in ``positive" direction.

Let $\gamma\subset V$ be a curve in the family $\Gamma$.
Let $x_\gamma$ be the closest point of $\gamma \cap J_1$ (in the order of
$J_1$) to $J_1 \cap \CC$, and let $y_\gamma$ be the closest point of $I_1
\cap \gamma$ (in the order of $I_1$ to $x = I_1 \cap \gamma$.
See Figure~\ref{f.P_1}.

\begin{figure}[phtb]
\centering
\psfrag{C}{$\CC$}
\psfrag{g}{$\gamma$}
\psfrag{xg}{$x_\gamma$}
\psfrag{yg}{$y_\gamma$}
\psfrag{I1}{$I_1$}
\psfrag{J1}{$J_1$}
\includegraphics[height=5cm]{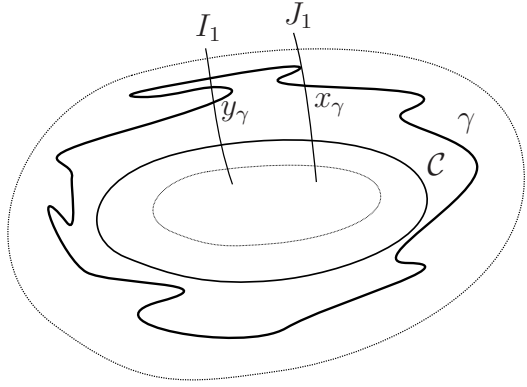}
\caption{The map $P_1$\;.} \label{f.P_1}
\end{figure}

Define $P_1(y_\gamma) := x_\gamma$.
This is a map from $\{ y_\gamma \colon \gamma \in \Gamma \}\subset I_1$ to
$\{ x_\gamma \colon \gamma \in \Gamma \}\subset J_1$.
This map is non-decreasing since for $\gamma, \eta \in \Gamma$, $\gamma
\neq \eta$, if $y_\gamma <_{I_1} y_\eta$ then $x_\gamma <_{J_1} x_\eta$
(since each curve of $\Gamma$ separates $V$ in exactly two components and
the curves in $\Gamma$ are disjoint).

Since the map is non-decreasing it may be extended to a
non-decreasing map $P_1 \colon I_1 \to J_1$. In the same say we
define $P_2 \colon I_2 \to J_2$, $\varphi \colon I_1 \to I_2$, and
$\psi \colon J_1 \to J_2$.

Notice that for $\gamma \in \Gamma$ and $x_\gamma$ as before we have
$$
\psi \circ P_1 (y_\gamma) = P_2 \circ \varphi(y_\gamma).
$$
Let $l_i \colon S_i' \to S_i$ defined by
$$
l_i = \Pi^u_f(\hat{p}_i,\hat{w}_i) \circ
\Pi^s_f(\hat{z}_i,\hat{p}_i) \circ\Pi^u_f(\hat{x},\hat{z}_i) \circ
P_i \circ \Pi^s_f(\hat{w}_i,\hat{x})|S_i'.
$$
We claim that $l_i$ is non-decreasing.
Indeed, it is equivalent to prove that
$$
\tilde{l}_i = \Pi^s_f(\hat{w}_i,\hat{x}) \circ
\Pi^u_f(\hat{p}_i,\hat{w}_i) \circ \Pi^s_f(\hat{z}_i,\hat{p}_i)
\circ \Pi^u_f(\hat{x},\hat{z}_i) \circ P_i
$$
is non-decreasing, which is equivalent to show that
$$
\hat{l}_i = \Pi^s_f(\hat{w}_i,\hat{x}) \circ
\Pi^u_f(\hat{p}_i,\hat{w}_i) \circ \Pi^s_f(\hat{z}_i,\hat{p}_i)
\circ \Pi^u_f(\hat{x},\hat{z}_i)
$$
from $J_i$ to $I_i$ preserves orientation (since it is a diffeomorphism).
Set $x_i=J_i\cap\ \mathcal{C},$ we know that $\hat{l}_i(x_i)=x.$
If $\hat{l}_i$ reverse orientation, then for $y>_{J_i} x_i$ we have that
$\hat{l}_i(y)<_{I_i} x$.
Since $\mathcal{C}$ is essential in $U$ we get that the accessibility
class of $C_0(y)$ must intersect $\mathcal{C}=C_0(x)$ and hence
$\mathcal{C}$ is not an accessibility class, a contradiction.

Now, let $\varepsilon >0$ be given and let $\delta$ be as in
Lemma~\ref{l.pert_sympl}.
For $|s|< \delta$ and $|t|<\delta$ we choose $h_1$ and $h_2$ as in
Lemma~\ref{l.pert_sympl} so that in $V_1'$ we have $h_1(y) = y+v_1$, $v_1$
in the direction of $S_1$, $\| v_1 \| = |s|$ and in $V_2'$ we have
$h_2(y)= y+v_2$, $v_2$ in the direction of $S_2$, $\| v_2 \| = |t|$.
So, $S_1$ is invariant by $h_1$ and $S_2$ is invariant by $h_2$, and
parametrizing $S_1$ and $S_2$, these maps have the form $h_{1/S_1}(y) = y+s$
and $h_{2/S_2}(y) = y + t$.

Now define $g = h \circ f$ where
$$
h = \left\{ \begin{array}{cl}
h_1(x) & \text{ if } x\in U_1, \\
h_2(x) & \text{ if } x\in U_2, \\
x & \text{ otherwise}.
\end{array}
\right.
$$
Notice that $\Pi^u_g(\hat{p}_i, \hat{w}_i) = h_i \circ
\Pi^u_f(\hat{p}_i, \hat{w}_i)$,
$\Pi^s_g(\hat{z}_i, \hat{p}_i) = \Pi^s_f(\hat{z}_i, \hat{p}_i)$,
$\Pi^u_g(\hat{x}, \hat{z}_i) = \Pi^u_f(\hat{x}, \hat{z}_i)$, and
$\Pi^s_g(\hat{w}_i, \hat{x}) = \Pi^s_f(\hat{w}_i, \hat{x})$.

Now, by Proposition~\ref{p.bb}, we may choose $s,t$ so that if
$q$ is a fixed point of
$$
h_1 \circ l_1 = \Pi^u_g(\hat{p}_1,\hat{w}_1) \circ
\Pi^s_g(\hat{z}_1,\hat{p}_1) \circ\Pi^u_g(\hat{x},\hat{z}_1) \circ
P_1 \circ \Pi^s_g(\hat{w}_1,\hat{x})|S_1'
$$
then $\left[ \Pi^s_f(\hat{x},\hat{w}_1) \circ \varphi \circ
\Pi^s_f(\hat{w}_2,\hat{x}) \right]^{-1}(q)$ does not contain any
fixed point of $h_2 \circ l_2$.

Thus, by conjugacy with $\Pi^s_g(\hat{x},\hat{w}_1)$ we have that if
$q$ is a fixed point of
$$
\hat{l}_1 := \Pi^s_g(\hat{w}_1,\hat{x}) \circ
\Pi^u_g(\hat{p}_1,\hat{w}_1) \circ \Pi^s_g(\hat{z}_1,\hat{p}_1)
\circ\Pi^u_g(\hat{x},\hat{z}_1) \circ P_1
$$
then $\varphi^{-1}(q)$ does not contain any fixed point of
$$
\hat{l}_2 := \Pi^s_g(\hat{w}_2,\hat{x}) \circ
\Pi^u_g(\hat{p}_2,\hat{w}_2) \circ \Pi^s_g(\hat{z}_2,\hat{p}_2)
\circ\Pi^u_g(\hat{x},\hat{z}_2) \circ P_2.
$$

Now, let $\gamma \in \Gamma, \gamma\subset V$ and let $y_\gamma, x_\gamma$
as before.
Then $\hat{l}_1(y_\gamma) \in C_0(x_\gamma,g) = C_0(P_1(y_\gamma),g)$.
So, if $y_\gamma$ is not fixed by $\hat{l}_1$, we have two
possibilities: either
\begin{itemize}
\item[(1)] $\hat{l}_1(y_\gamma) <_{I_1} y_\gamma$ or
\item[(2)] $\hat{l}_1(y_\gamma) >_{I_1} y_\gamma$.
\end{itemize}
In case (1), $\hat{l}_1(y_\gamma)$ cannot belong to $\gamma$ by the
definition of $y_\gamma$ and so $C_0(x_\gamma,g)$ is not contained in
$\gamma$.

In case (2), we conclude that the point
$$
z:= \Pi^u_g(\hat{z}_1, \hat{x}) \circ \Pi^s_g(\hat{p}_1, \hat{z}_1)
\circ \Pi^u_g(\hat{w}_1, \hat{p}_1) \circ \Pi^s_g(\hat{x},
\hat{w}_1) (y_\gamma)
$$
satisfies $z <_{J_1} x_\gamma$ and so does not belong to $\gamma$
which implies that $C_0(y_\gamma,g)$ is not contained in $\gamma$.

Finally, assume that $y_\gamma$ is fixed by $\hat{l}_1$ and let
$\bar{x}_\gamma$ be the closest point of $\gamma \cap I_2$ (in the order
of $I_2$) to $x$.
Then we know that $\bar{y}_\gamma$ is not fixed by $\hat{l}_2$ and we
apply the previous argument.
Thus, no curve $\gamma \in \Gamma$ is an accessibility class.
The proof is finished.
\end{proof}

\section{Proof of Theorem \ref{trivialaccesibility}}
\label{s.theorem-trivialaccesibility}
Let $r\ge 2$ and let $\EE$ be as in Theorem \ref{trivialaccesibility},
i.e., $\EE$ is $\EE^r, \EE^r_m, \EE^r_\omega, \EE^r_{sp}$ or
$\EE^r_{sp,\omega}$.
We have to prove the set $\RR_0$ of diffeomorphisms in $\EE$ having no
trivial accessibility classes is $C^1$ open and $C^r$ dense.
This result is a consequence of Lemma \ref{l.primeiro}, as follows.

Lets consider $\Gamma_0 \colon \EE \to \CC(M)=\{$compact subsets of $M \}$
(endowed with the Hausdorff topology)
\begin{equation}\label{eq-gamma0}
\Gamma_0 (f) = \{ x \in M \colon AC(x) \text{ is trivial} \}.
\end{equation}
We observe that $\Gamma_0(f)$ is indeed a compact set, it follows from
Corollary~\ref{nontrivial}.

\begin{lemma}
The map $\Gamma_0$ is upper semicontinuous, i.e., given $f \in \EE$ and a
compact set $K$ such that $\Gamma_0(f) \cap K = \emptyset$ then there
exists a neighborhood $\UU(f)$ of $f$ in $\EE$ (which is also $C^1$ open)
such that $\Gamma_0(g) \cap K = \emptyset$ for all $g \in \UU(f)$.
\end{lemma}

\begin{proof}
Let $y \notin \Gamma_0(f)$. From Corollary~\ref{nontrivunif} there exist
$U(y)$ and $\UU_y(f)$ (which is also $C^1$  open)  such that for any
$g \in \UU_y(f)$ and $z \in U(y)$ we have that $AC(z,g)$ is non-trivial.

Now, consider the family of $U_y$ with $y\in K$. We may cover $K$ with
finitely many of them, say $K\subset \bigcup_{i=1}^n U_{y_i}$.

Let $\UU(f) \subset \bigcap_{i=1}^n \UU_{y_i}(f)$. If $g \in \UU(f)$ and
$z\in K$ then $z\in U_{y_i}$ and $g\in \UU_{y_i}(f)$ for some $i$ and so
$AC(z,g)$ is non-trivial. The proof of the lemma is complete.
\end{proof}

By taking $K = M$ in the previous lemma, we get:

\begin{corollary}
\label{c.1} If for some $f\in\EE$ we have that $\Gamma_0(f) = \emptyset$
then there is a neighborhood $\UU(f)\subset \EE$ (which is $C^1$ open)
such that for any $g \in \UU(f)$ we have $\Gamma_0(g) = \emptyset$.
\end{corollary}

Now, we are ready to conclude:
\begin{proof}[Proof of Theorem~\ref{trivialaccesibility}.]
Let $\GG_0$ be the set of continuity points of $\Gamma_0$.
This is a residual set in $\EE$ (since $\EE$ with the $C^r$ topology is a
Baire space).
We claim that if $f\in \GG_0$ then $\Gamma_0(f) = \emptyset$.
Otherwise, let $x \in \Gamma_0(f)$ and we may assume that $x$ belongs to a
periodic compact center leaf (see Lemma \ref{l.4}).

Indeed, by the continuity of $\Gamma_0$ at $f$ we have that for any
neighborhood $V$ of $x$ there exists $\UU(f)$ such that for any $g \in
\UU(f)$ there is $x_g \in V$ such that $AC(x_g,g)$ is trivial.
A direct application of Lemma~\ref{l.primeiro} yields a contradiction and
the claim is proved.

From this and Corollary~\ref{c.1} we get that the set
\begin{equation}
\label{eq.R0}
\RR_0=\{f\in\EE: \Gamma_0(f) = \emptyset\}
\end{equation}
is $C^1$ open and $C^r$ dense in $\EE.$
This set $\RR_0$ is just the set of diffeomorphisms where any accessibility
class is nontrivial.
\end{proof}

\section{The accessibility class of periodic points}
\label{s.pp}

Through this section, we consider $\EE$ to be either $\EE^r, \EE^r_m,
\EE^r_\omega, \EE^r_{sp}$ or $\EE^r_{sp,\omega}$ and with $\dim E^c=2.$

Let $f\in\EE$ and let $\FF^c_1$ be a periodic center leaf of period $k$
and let $p\in \FF^c_1$ be a periodic point of $f$.
We will classify the periodic points with respect to its behaviour on
the central leaf. In particular we say
\begin{itemize}
\item $p$ is \emph{center-hyperbolic of saddle type} if $p$ is hyperbolic
of saddle type with respect to $f^k_{/\mathcal{F}^c_1}$,
\item $p$ is \emph{center-attractor} or \emph{center-repeller} if $p$ is
attractor or repeller w.r.t. $f^k_{/\mathcal{F}^c_1}$,
\item $p$ is \emph{center-elliptic} if $p$ is elliptic w.r.t.
$f^k_{/\mathcal{F}^c_1}$.
\end{itemize}
Assume that $p$ is a center-hyperbolic periodic point of saddle type
$f^k_{/\mathcal{F}^c_1}$.
We denote by $CW^s(p)$ the {\em stable manifold of $p$ with respect to}
$f^k_{/\FF^c_1}$.
We write $CE^s_p \subset T_p\FF^c_1$ the tangent space to $CW^s(p)$.
Analogously, we denote $CW^u(p)$ the {\em unstable  manifold of $p$ with
respect to} $f^k_{/\FF^c_1}$ and
$CE^s_p \subset T_p\FF^c_1$ the tangent space to $CW^s(p)$.

If $p$ is a periodic point of $f$, $p \in \FF^c_1$, and $U$ is a
neighborhood of $p$ in $M$, we denote by $C_0(p,U,f)$ the {\em local
accessibility class of} $p$, that is, the set $y \in \FF^c_1$ that can be
joined to $p$ by $su$-path contained in the neighborhood $U$ of $p$.

We say that a periodic point $p$ of period $\tau(p)$ of $f$ is
{\em generic}  if:
\begin{itemize}
\item $p$ is hyperbolic in the case $\EE=\EE^r$ or $\EE^r_{sp}.$
\item $-1$ and $1$ are not eigenvalues of $Df^{\tau(p)}_p$ in the case
$\EE=\EE^r_m, \EE^r_\omega$ or $\EE^r_{sp,\omega}.$
\end{itemize}

\begin{lemma}
\label{l.cc}
There exists a residual set $\GG_1$ in $\EE$ such that if $f \in \GG_1$
and $p$ is a center-hyperbolic periodic point of saddle type of $f$ for
$f^k_{/\FF^c_1}$ then there exist neighborhoods $U_c$ and $U$ of $p$, $U_c$
in $\FF^c(p)$ and $U$ in $M$, such that $U_c \setminus (CW^s_{loc}(p) \cup
CW^u_{loc}(p))$ has four connected component, $U \cap \FF^c(p) \subset
U_c$, and $C_0(p,U,f)$ is not contained in $CW^s_{loc}(p) \cup
CW^u_{loc}(p)$.
\end{lemma}

We say that a periodic point $p$ as in the previous lemma satisfies the
{\em Property $(L)$}.

\begin{proof}
Let $H_n = \{ f \in \EE \colon$ all points in $\Fix(f^n)$ are generic and
every center-hyperbolic periodic point $p \in \Fix(f^n)$ of saddle type
satisfies Property~$(L)\}$.

\noindent{\em Claim:} $H_n$ is open and dense in $\EE$.
In fact, notice that $H_n^0 = \{ f\in \EE \colon \Fix(f^n)
\text{ generic}\}$ is open and dense in $\EE$.
Thus, to prove the claim it is enough to show that $H_n$ is open and dense
in $H_n^0$.
It is immediate that $H_n$ is open in $H_n^0$.
Let us show that $H_n$ is dense.
Let $f\in H^0_n$.
We know that there are finitely many center-hyperbolic periodic points in
$\Fix(f^n)$.
Choose a neighborhood $U_c$ for each one as in Property $(L)$.
By similar arguments as Lemma~\ref{l.primeiro} it is not difficult to get
$g \in H_n$ arbitrarily close to $f$ satisfying Property $(L)$.
Finally, set $\GG_1 = \cap_{n\ge 0} H_n$ and the lemma is proved.
\end{proof}

\begin{theorem}
\label{t.peridic_point}
There exists a residual subset $\RR_*$ in $\EE$ such that if $f\in \RR_*$
and $p$ is a periodic point which is neither a center-attractor nor a
center-repeller for $f^k_{/\FF^c_1}$ on a compact periodic center leaf
$\FF^c_1$ then $C_0(p,f)$ is open.
Moreover, if $p$ is center-hyperbolic of saddle type then there exist an
open set $V$ in $M$ (contained in a ball around $p$) and a neighborhood
$\UU(f)$ (which is also $C^1$ open) such that for any $g \in \UU(f)$, we
have $V \subset AC(p_g,g)$ where $p_g$ is the continuation of $p$ for
$g\in \UU(f)$.
\end{theorem}

\begin{proof}
Let $\RR_* = \RR_0 \cap \GG_1$ where $\RR_0$ is as in
Theorem~\ref{trivialaccesibility} (see also \eqref{eq.R0}).
Let $f \in \RR_*$ and let $p$ be a periodic point of $f$.
Since $f \in \RR_*$, $C_0(p):=C_0(p,f)$ is either open or a one dimensional
$C^1$-submanifold.

Assume first that $p$ is a center-elliptic periodic point of
$f^k_{/\FF^c_1}$ of period $\tau(p)$.
Now, since $C_0(p)$ is invariant under $f^{\tau(p)}$ and there is no
invariant direction of $Df^{\tau(p)}_{/\FF^c(p)}$ we easily conclude that
$C_0(p)$ is open.

Assume now that $p$ is a center-hyperbolic periodic point of
$f^k_{/\FF^c_1}$ of saddle type of period $\tau(p)$ and assume, by
contradiction, that $C_0(p)$ is not open, that is, $C_0(p)$ is a one
dimensional $C^1$-submanifold.
Then $p$ satisfies Property $(L)$.
This implies that there exist non trivial connected set $C \subset
C_0(p,U,f)\subset C_0(p)$ and not contained in $CW^s_{loc}(p) \cup
CW^u_{loc}(p)$.

On the other hand, $T_pC_0(p)$ must be an invariant direction (by the
invariance of $C_0(p)$) by $Df^{\tau(p)}_p$ and so $T_pC_0(p) = CE^s_p$ or
$CE^u_p$.
Assume that $T_pC_0(p) = CE^s_p$.
Thus, $C_0(p)$ is locally a graph around $p$ (via the exponential map) of
a map from $CE^s_p \to CE^u_p$.

Now, this graph is not contained in $CW^s_{loc}(p)$ (since $p$ satisfies
Property~$(L)$ and the connected set $C  \subset C_0(p,U,f))$.
But notice that this graph is locally invariant, by the invariance of
$C_0(p)$.
This is a contradiction since there is a unique locally invariant graph,
namely $CW^s_{loc}(p)$.
Analogously, if $T_pC_0(p) = CE^u_p$ we use the same argument for $f^{-1}$.
Thus, we have proved that $C_0(p)$ is open.

As the accessibility class $C_0(p)$ of a center-hyperbolic periodic point
of saddle type $p$ is open and that $CW^s_{loc}(p)$ and $CW^u_{loc}(p)$
intersect $C_0(p)$ then, by invariance we get that $CW^s(p)$ and $CW^u(p)$
are contained in $C_0(p)$.

From Property $(L)$, we know that there exists $y_f \in U_c \setminus
CW^s_{loc}(p) \cup CW^u_{loc}(p)$ so that $y_f \in C_0(p,U,f)$.
Lets order the four connected component clockwise beginning with the one
that contains $y_f$.
See Figure~\ref{f.pert_open}.

\begin{figure}[phtb]
\centering
\psfrag{p}{$p_f$}
\psfrag{I}{$(I)$}
\psfrag{II}{$(II)$}
\psfrag{III}{$(III)$}
\psfrag{IV}{$(IV)$}
\psfrag{Vc}{$V^c$}
\psfrag{Bc}{$B^c(p_f)$}
\psfrag{Bcy}{$B^c(y_f)$}
\psfrag{y}{$y_f$}
\psfrag{U}{$U_c$}
\includegraphics[height=9cm]{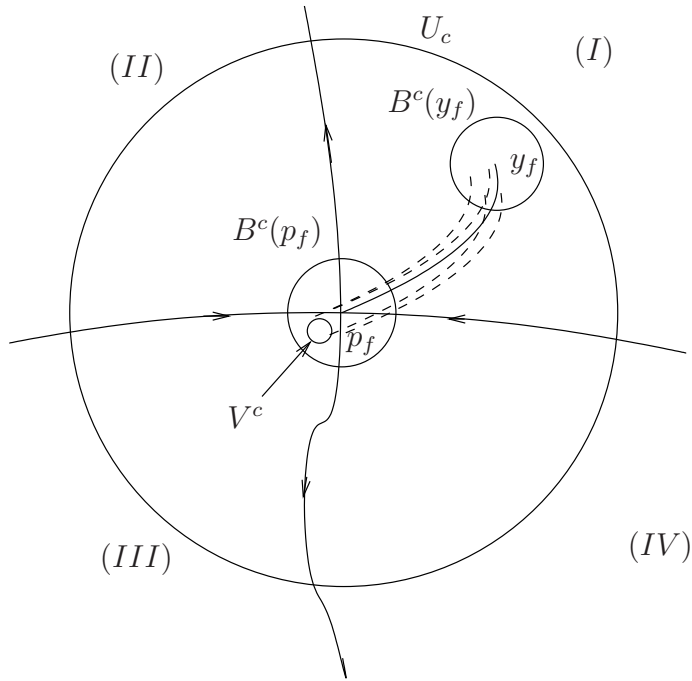}
\caption{Accessibility class of hyperbolic periodic points.}
\label{f.pert_open}
\end{figure}

Let $B_c(y_f)$ be a ball centered in $y_f$ contained in $U_c \setminus
(CW^s_{loc}(p) \cup CW^u_{loc}(p))$, that is, $B_c(y_f) \subset (I)$.

We know (see Lemma~\ref{l.path}) that there exists a continuous map
$\gamma \colon B_c(p) \times [0,1] \to U_c$ such that
\begin{itemize}
\item[$\cdot$] $\gamma(z,0) = z$,
\item[$\cdot$] $\gamma(z,1) \subset B_c(y_f)$, and
\item[$\cdot$] $\gamma(z,t) \subset C_0(z,U,f)$.
\end{itemize}

Let $V_c$ be an open set such that $\overline{V_c} \subset B_c(p) \cap
(III)$.
Thus, for any $z\in V_c$ we have for some $t_0$ that $\gamma(z,t_0) \in
(CW^s_{loc}(p) \cup CW^u_{loc}(p)) \subset C_0(p)$ and so
$V_c \subset C_0(p)$.

Finally, we saturate $V_c$ by local (strong) stable and unstable manifolds
to obtain an open set $V \subset M$.
This set $V$ satisfies the requirement of the theorem for $\UU(f)$ small
enough by the continuation of center leaves, strong stable and unstable
leaves, the continuation of $p$, and the continuation of $\gamma$ and that
Property $(L)$ is open.
The theorem is proved.
\end{proof}

Let $f \in \RR_*$ and let $p$ be a center-hyperbolic periodic point of
saddle type (belonging to a periodic compact center leaf).
Let $\UU(f)$ corresponding to $p$ and $f$ in Theorem~\ref{t.peridic_point}
(we denote by $p_g$ the continuation of $p$ for $g \in \UU(f)$).

Let $\Gamma_1 \colon \UU(f) \to \CC(M)$,
\begin{equation}\label{eq.gamma1}
\Gamma_1(g) = M \setminus AC(p_g,g).
\end{equation}

Notice that $\Gamma_1$ is well defined since $AC(p_g,g)$ is open for
$g\in \UU(f)$.

\begin{proposition}
\label{p.gamma1upper}
The map $\Gamma_1$ is upper semicontinuous, i.e., given $g\in \UU(f)$ and
a compact set $K$ such that $\Gamma_1(g) \cap K = \emptyset$ then there
exists a neighborhood $\VV(g) \subset \UU(f)$ (which is also $C^1$ open)
such that for any $h \in \VV(g)$ we have that $\Gamma_1(h) \cap K =
\emptyset$.
\end{proposition}

\begin{proof}
Let $V$ be the fixed open set in $M$ from Theorem~\ref{t.peridic_point},
that is, $V \subset AC(p_g,g)$ for every $g \in \UU(f)$.
Let $g \in \UU(f)$ and $K$ compact with $\Gamma_1(g) \cap K = \emptyset$
be given.
Let $y \in K$, then there exists $U_y$ and $\UU_y(g)$ such that for any
$h \in \UU_y(g)$ and any $z \in U_y$ we have that $AC(z,h) \cap V \neq
\emptyset$ (see Corollary~\ref{nontrivunif}), in other words $U_y \subset
AC(p_h,h)$ for any $h \in \UU_y(g)$.

Now, cover $K$ with finitely many of these open sets $U_y$, that is,
$K \subset \bigcup_{i=1}^n U_{y_i}$. Let $\VV(g) = \bigcap_{i=1}^n
\UU_{y_i}(g)$.
Then, for every $h \in \VV(g)$ we have that $K \subset AC(p_h,h)$.
The proof of proposition in finished.
\end{proof}

\begin{corollary}\label{c.accopen}
Assume that for $g \in \UU(f)$ we have that $\Gamma_1(g) = \emptyset$,
then $g$ is accessible.
Moreover, there exists $\VV(g)$ (which is also $C^1$ open) such that any
$h \in \VV(g)$ is accessible.
\end{corollary}

\section{Proof of Theorem \ref{axa}}
\label{s.axa}

Let $\EE$ be either $\EE^r_A$ or $\EE^r_{A,m}$ and with $\dim E^c=2$, that
is, the set of $\EE^r$ or $\EE^r_m$ where $\dim E^c=2$ and supporting a
Center Axiom-A.
Let $\RR_0$ be as in Theorem~\ref{trivialaccesibility} and $\RR_*$ as in
Theorem~\ref{t.peridic_point} (both restricted to $\EE^r_A$ or
$\EE^r_{A,m}).$
Thus, $\RR_0\cap\RR_*$ is residual in $\EE.$
Let $f\in \RR_0\cap\RR_*$ (although in the proof of
Theorem~\ref{t.peridic_point} we construct $\RR_* \subset \RR_0$).
We will prove that $f$ is accessible. Lets see the properties we know for
$f:$
\begin{itemize}
\item Any accessibility class is nontrivial.
\item The accessibility classes of center-hyperbolic periodic points of
saddle type of $f$ are open.
\item If $p$ is a center-hyperbolic periodic point of saddle type, then it
satisfies Property $(L).$
\item There exists a compact periodic center leaf $\FF^c_1$ such that
$f^k_{/\FF^c_1}$ is an Axiom-A diffeomorphism without having both periodic
attractor and periodic repellers, where $k$ is the period of $\FF^c_1.$
We will assume for simplicity that it has no attractors.
\item If $p$ is center-hyperbolic periodic point of saddle type, then the
stable and unstable manifolds in the center leaf $CW^s(p)$ and $CW^u(p)$
are contained in the accessibility class of $p.$
\end{itemize}

Lets denote by $\Lambda$ a basic piece of $f^k_{/\FF^c_1}$ which is not a
periodic center-repeller.
Recall that the stable and unstable manifolds $CW^s(\OO(p))$ and
$CW^u(\OO(p))$ are dense in $\Lambda$ for any periodic point $p\in\Lambda.$
Let $x\in\Lambda$ be any point.
We know that $C_0(x)$ is open or a one dimensional $C^1$ manifold without
boundary containing $x.$
In any case, we have that it intersects $CW^s(\OO(p))$ or $CW^u(\OO(p))$
and therefore $C_0(x) \cap C_0(f^i(p)) \neq \emptyset$ for some $i.$
Therefore, $C_0(x) = C_0f^i(p)$.
This means that $\Lambda$ is contained in $\cup_{q\in\OO(p)}C_0(q)$ for
$p\in\Lambda$ periodic.
By the invariance of $\cup_{q\in\OO(p)}C_0(q)$ we also have that
$CW^s(\Lambda)$ and $CW^u(\Lambda)$ are contained in
$\cup_{q\in\OO(p)}C_0(q).$
Let $F_0$ be the set of periodic center-repellers which is a finite set.
Let $F_1$ be the set of center-hyperbolic periodic points of saddle type.
Since there are no periodic center-attractor we have that
$\Per(f_{/\FF^c_1}) = F_0 \cup F_1$.

Since every point in $\FF^c_1$ is contained in the stable manifold (inside
the center leaf) of the basic pieces, we conclude that
$\FF^c_1 \setminus F_0 \subset \cup_{p\in F_1}C_0(p)$ and that $C_0(p)$ is
open for any periodic point in $F_1$.
By connectedness we conclude that $\FF^c_1 \setminus F_0 \subset C_0(p)$
(for any periodic point $p \in F_1$).
Since the accessibility classes are non trivial for every $q \in F_0$, we
have that $C_0(q) \cap C_0(p) \neq \emptyset$ and so $C_0(q) = C_0(p)$.
Thus $\FF^c_1 = C_0(p)$ for any $p \in F_1$.
That is, the center leaf $\FF^c_1$ is just one center accessibility class
and by Lemma~\ref{l.4} we have that $f$ is accessible, as we claimed.

Finally, since $f\in\RR_*$ and fixing a center-hyperbolic periodic point
of saddle type $p$ of $f$, in the setting of Section~\ref{s.pp}, we have a
neighborhood $\UU(f)$ and a map $\Gamma_1$ defined in $\UU(f).$
Due to what we just proved  $\Gamma_1(f) = \emptyset$ holds and so by
Corollary~\ref{c.accopen} there exists $\UU_0(f)$ such that any
$g\in\UU_0(f)$ is accessible.
Thus,
$$
\RR_1=\bigcup_{f\in\RR_0\cap\RR_*}\UU_0(f)
$$
is $C^1$ open and $C^r$ dense in $\EE$ and formed by accessible
diffeomorphisms.
This completes the proof of Theorem~\ref{axa}.
\hfill $\Box$

\subsection{Examples}

We present here some examples where Theorem \ref{axa} applies.

\begin{example}
This example can be thought as a conservative version of the well known
Shub's example on $\mathbb{T}^4$ (\cite{Sh71}).

Consider the Lewowicz family (see \cite{Le80}) of conservative
diffeomorphisms on $\T^2:$
$$
f_c(x,y)=\left(2x -\frac{c}{2\pi}\sin2\pi x+y, x- \frac{c}{2\pi}\sin2\pi
+y\right),\,\,\,c\in\mathbb{R}.
$$
Notice that when $c=0$ $f_c$ is Anosov and when $1<c<5$ the fixed point at
$(0,0)$ is an elliptic fixed point.
We just consider for instance $c\in[0,2].$
From this family it is not difficult to construct a continuous map
$$
g:\T^2\to \diff^\infty(\T^2)
$$
such that for two points $p,q \in \mathbb{T}^2$ given, we have $g(p)=f_0$
and $g(q)=f_{2}.$
Now, given $r\ge 2,$ consider a conservative Anosov diffeomorphism
$A:\T^2\to\T^2$ having $p,q \in \mathbb{T}^2$ as fixed points and with
enough strong expansion and contraction so the map
$$
F:\T^2\times\T^2\to\T^2\times\T^2 \mbox{ defined as }F(x,y)=(A(x),g_x(y))
$$
belongs to $\EE^r_{A,m}(\T^2\times\T^2).$
The center foliation is thus $\{x\}\times\T^2$ and at $\FF^c(p)=\{p\}
\times \T^2$ the map $F$ supports an Anosov (and hence an Axiom-A without
having both periodic attractors and repellers).
Our theorem implies that a generic arbitrarily small $C^r$ perturbation
(preserving the Lebesgue measure on $\T^2\times\T^2$) is stably ergodic.
The same example can be considered also just in the skew-product setting.
\end{example}

\begin{remark}
Notice that due to the presence of an elliptic point on $\{q\}\times\T^2$
the center bundle $E^c$ does not admit any dominated splitting.
By the result in \cite{BFP06} we may find a perturbation of $F$ (and stably
ergodic) such that the center Lyapunov exponent is nonzero.
This implies that the center foliation (which is two dimensional) is not
absolutely continuous.
\end{remark}

Theorem~\ref{axa} admits some generalizations or different versions.
We just give some examples and an idea of how to prove stable ergodicity.

\begin{example}
Consider $f:\T^2\to\T^2$ a conservative Anosov diffeomorphism and let
$F_0:\T^2\times \mathbb{S}^1\to \T^2\times \mathbb{S}^1$ as
$F_0=f\times id.$
This is a conservative partially hyperbolic diffeomorphism on $\T^3$ with
one dimensional center and the center foliation is by circles.
Let $p$ be a fixed point of $f.$
It is not difficult to construct a (conservative) perturbation $F$ of
$F_0$ such that on the corresponding $\FF^c(p,F)$ the dynamics is a north-south
Morse-Smale dynamics and $F$ satisfies the same generic conditions as in
Theorem~\ref{axa}.
Then, the map $F\times F:\T^6\to\T^6$ belongs to $\EE^r_{m}(\T^6)$, the
center foliation is by $\T^2$ and $F_{/\FF^c(p,p)}$ is an
Axiom-A diffeomorphism (the product of the two Morse-Smale on the circle).
Theorem~\ref{axa} does not apply in this case because $F_{/\FF^c(p,p)}$
have a center-attracting and a center-repelling periodic point.
Nevertheless, by the similar arguments as in the proof of Theorem~\ref{axa}
one gets that the union of the accessibility classes of the two
center-hyperbolic saddles in $\FF^c(p,p)$ is open and if it is not the
whole center leaf $\FF^c(p,p)$ then its complement consist of a closed
$C^1$ curve which is a connection between the attractor and the repeller.
Since this curve does not separate the leaf $\FF^c(p,p)$ we have that the
union of the accessibility classes of the two center-hyperbolic saddles is
just one accessibility class $C_0(q)$, for $q$ any periodic saddle, which
is open and $AC(q)$ has full measure in $\T^6$.
This means that $F$ is essentially accessible and hence ergodic.
Since the above situation is $C^r$ open we conclude that $F$ is $C^r$
stably ergodic.
\end{example}

The same argument also applies to next example.

\begin{example}
Consider $f_0:M\to M$ to be the time one map of the suspension of a
conservative Anosov diffeomorphisms on $\T^2.$
This is a conservative partially hyperbolic diffeomorphism whose leaves of
the center foliation are the orbits of the suspension.
Let $p$ be a fixed point of the Anosov and the center leaf $\FF^c(p,f_0)$
is a circle where $f_0$ is the identity.
We then find a conservative and generic perturbation $f$ of $f_0$ such
that $f$ restricted to $\FF^c(p,f)$ is a Morse-Smale system.
The diffeomorphism $f\times f:M\times M\to M\times M$ belongs to
$\EE^r_{m}(M\times M)$ with two dimensional center leaves and in the
leaf $\FF^c((p,p),f\times f)$ is an Axiom-A on a two torus and in the same
situation as the previous example.
The same argument yields that $f \times f$ is stably ergodic.
\end{example}

\section{Proof of Theorem \ref{main}}\label{s.main}

In this section we denote by $\EE$ either $\EE^r_\omega$ or
$\EE^r_{sp,\omega}$ and with $\dim E^c=2.$
The key fact about preserving a symplectic form $\omega$ in $\EE^r_\omega$
is the following folklore result (see \cite[Lemma 2.5]{Xia06}):

\begin{lemma}
Let $f:M\to M$ be a partially hyperbolic diffeomorphism preserving a
symplectic form $\omega.$
Then $\omega_{/E^c}$ is non degenerated (and so symplectic).
In particular if $f$ is dynamically coherent and $\dim E^c=2$ then $\omega$
is an area form in $\FF^c(x)$ for any $x.$
Furthermore, if $\FF^c(x)$ is $k$-periodic then $f^k_{/\FF^c(x)}$ is a
conservative diffeomorphism.
\end{lemma}

Throughout this section in order to simplify notation we omit the word
center when we refer to the classification of periodic point in a center
leaf in Section~\ref{s.pp}.

The next lemma says that generically we have compact leaves with periodic
points.

\begin{lemma}
There exists a $C^1$ open and $C^r$ dense set $\GG_2$ in $\EE$ such that if
$f\in\GG_2$ then there exists a periodic compact leaf having a hyperbolic
periodic point.
\end{lemma}

\begin{proof}
Notice that the set in $\EE$ having a hyperbolic periodic point on some
compact periodic leaf is $C^1$ open.

Let $f_0\in\EE$ and let $\FF^c_1$ be a compact periodic leaf.
For simplicity we assume it is fixed.
We may assume also that $\FF^c_1$ is orientable (otherwise we go to the
double covering).
If $\FF^c_1$ is not the two torus then $f_{0/\FF^c_1}$ has periodic points.
Let $f\in\EE$ be a Kupka-Smale diffeomorphism and arbitrarily $C^r$ close
to $f_0.$
It follows that there is a hyperbolic periodic point in $\FF^c_1(f)$ since
once we have elliptic periodic points we have hyperbolic periodic points
(see \cite{Z73}).

If $\FF^c_1$ is the two torus and $f_0$ has no periodic points in
$\FF^c_1$, then by composing with a translation in the torus (and extending this  perturbation on $\FF^c_1$ to $M$) we may change
the mean rotation number (the mean rotation number of the composition of
two maps is the sum of the mean rotation number of each one) to get a
rational mean rotation number.
Using  a result by Franks \cite{Fra88} we get a periodic point, which by
perturbation we may assume that it is hyperbolic or elliptic.
And then we argue as before.
\end{proof}

\begin{remark}
In this situation we are working with ($\EE^r_\omega$ or
$\EE^r_{sp,\omega}$), if $p$ is a hyperbolic periodic point of $f$ then
for the restriction to the center manifold that contains $p$ we have that
$p$ is a hyperbolic periodic point of saddle type (since the restriction
to the center manifold preserves area).
\end{remark}

Let $\RR_0$ from Theorem~\ref{trivialaccesibility}, and let $\KS$ be the
set ok Kupka-Smale diffeomorphisms in $\EE$ which is a residual set, and
let $\RR_*$ as in Theorem~\ref{t.peridic_point}.
We consider
\begin{equation}
\label{eq.generic}
f\in\RR_0\cap\RR_*\cap\GG_2\cap\KS
\end{equation}
and let $\FF^c_1$ be a compact
$k$-periodic leaf containing a hyperbolic periodic point
$p,\,\FF^c_1=\FF^c(p).$
Let $\Gamma_1$ as in \eqref{eq.gamma1}.

Let $\UU(f)$ be as in Theorem \ref{t.peridic_point}.
Let $\tilde{\RR} \subset
\UU(f)$ be the residual subset of continuity points of $\Gamma_1$ (as
defined in \eqref{eq.gamma1}).

We will assume for simplicity that the compact leaf $S=\FF^c(p,f)$ is
fixed by $f.$
When $g$ varies on $\UU(f)$ the compact leaf $\FF^c(p_g,g)$ varies
continuously (by a homeomorphism on $M$ close to the identity), and thus
there is a natural identification between $\FF^c(p_g,g)$ with $\FF^c(p,f)$
as the surface $S.$
In order to avoid unimportant technicalities we will assume that
$\FF^c(p_g,g)=S$ for any $g\in\UU(f).$
Now consider the following maps $\Gamma_2, \Gamma_3:\UU(f)\to \CC(S)$,
where $\CC(S)$ is the set of compact subset of $S$ with the Hausdorff
topology:
\begin{equation}\label{eq.gamma2}
\Gamma_2(g)=S \setminus C_0(p_g,g) \qquad \text{ and } \qquad
\Gamma_3(g)=\overline{C_0(p_g,g)},
\end{equation}
where $C_0(p_g,g)$ is the connected component of $C(p_g,g)$ that contains
$p_g.$

\begin{lemma} \label{l.gamma2y3}
The map $\Gamma_2$ is upper semicontinuous and the map $\Gamma_3$ is lower
semicontinuous.
That is, given $g\in\UU(f)$, a compact set $K\subset S$ and an open set
$U\subset S$ such that $K\cap \Gamma_2(g)=\emptyset$ and $U\cap\Gamma_3(g)
\neq\emptyset$ then there exists $\UU(g)$ such that $K\cap \Gamma_2(h)=
\emptyset$ and $U\cap\Gamma_3(h)\neq\emptyset$ for any $h\in\UU(g).$
\end{lemma}

\begin{proof}
The proof that $\Gamma_2$ is upper semicontinuous is similar as the proof of
Proposition~\ref{p.gamma1upper}.
Let $V'=V\cap S$ where $V$ is as in Theorem~\ref{t.peridic_point} and let
$K\subset S$ a compact set as in statement, that is $K\subset C_0(p_g,g).$
Let $y\in K.$
There exists $U_y\subset S$ and $\UU_y(g)$ such that for any $z\in U_y$
and $h\in\UU_y(g)$ we have that $AC(z,h)\cap V'\neq\emptyset.$
On the other hand we can assume that $U_y\subset C_0(p_g,g)$ and this
means that the $su$-path (of $g$) joining $z\in U_y$ with $V'$ when lifted
to the covering $\tM$ is a path that starts and ends on a same center leaf
(which projects to $S$) and so the same happens for $h$ near $g.$
This implies that $U_y\subset C_0(p_h,h)$ for any $h\in\UU_y(g).$
Then, covering $K$ with finitely many sets $U_y$ and taking the
corresponding intersection of the $\UU_y(g)$ we conclude the statement on
$\Gamma_2$.

Lets prove the semicontinuity of $\Gamma_3$.
Let $U\subset S$ be an open set such that $U\cap\Gamma_3(g)\neq\emptyset.$
In particular $U\cap C_0(p_g,g)\neq\emptyset.$
Let $y$ be in this intersection and let $U_y$ and $\UU_y(g)$ as before.
We may assume that $U_y\subset U.$
Then, for any $h\in\UU_y(g)$ we have that $U_y\subset C_0(p_h,h)$ and so
$U\cap \Gamma_3(h)\neq\emptyset.$
\end{proof}

Let $\RR_2(f)$ and $\RR_3(f)$ be the sets of continuity points of
$\Gamma_2$ and $\Gamma_3$ respectively.
These are residual subsets of $\UU(f).$
We set
\begin{equation}
\label{eq.generic-neigh}
\RR_{\UU(f)}=\RR_0 \cap\GG_2\cap \RR_* \cap \tilde{\RR} \cap \RR_2 (f)
\cap \RR_3(f) \cap \KS\cap\UU(f),
\end{equation}
which is a residual subset of $\UU(f).$ The next result implies our
Theorem~\ref{main}.

\begin{theorem}\label{t.accesibleRR}
Let $g\in\RR_{\UU(f)}.$ Then  $C_0(p_g,g) = \FF^c(p_g,g)$, i.e., $g$
is accessible.
\end{theorem}

Indeed, for any $f$ as in \eqref{eq.generic} we consider $\RR_{\UU(f)}$ as
in \eqref{eq.generic-neigh} and we set
$$
\RR=\bigcup_{f\in\RR_0\cap\RR_*\cap\GG_2\cap\KS}\RR_{\UU(f)}.
$$
It follows that $\RR$ is residual (and hence $C^r$ dense) in $\EE$ and
every $g\in\RR$ is accessible from Theorem~\ref{t.accesibleRR}.
On the other hand, Corollary \ref{c.accopen} implies that the accessible
ones are $C^1$ open.

The rest of the section is thus devoted to prove Theorem~\ref{t.accesibleRR}.

\begin{lemma}
Let $g \in \UU(f)$ and let $K$ be a connected component of the boundary
$\partial C_0(p_g,g)$.
Then, for every $x \in K$ we have that $C_0(x) \subset K$.
\end{lemma}

\begin{proof}
Let $x\in K$ and let $y\in C_0(x)$ and assume that $y \notin K$.
Since $C_0(x)$ is connected, we may assume, without loss of
generality, that $y \notin \partial C_0(p_g,g)$. Therefore, since $y$
cannot belong to $C_0(p_g)$, we have that $y \notin
\overline{C_0(p_g)}$.

On the other hand we know (by Lemma \ref{l.path} and Remark~\ref{path2})
there is a continuous map $\gamma \colon B_x \to B_y$, such that for any
$z \in B_x$, $\gamma(z) \in C_0(z)$, where $B_x$ and $B_y$ are small
neighborhoods $x$ and $y$, respectively, and we may take $B_y \subset
\FF^c(p_g,g) \setminus \overline{C_0(p_g)}$.
Since $x\in \partial C_0(p_g)$ then we may take $z \in B_x \cap C_0(p_g)$
and hence $\gamma(z) \in C_0(p_g) \cap B_y$, a contradiction.
\end{proof}

\begin{remark}
Let $K$ be a connected component of $\partial C_0(p_g,g)$.
Then $K$ has no periodic point.
This is because $K$ has empty interior and we know that for any periodic
point $q$ of $g$, $C_0(q,g)$ is open.
\end{remark}

\begin{lemma}
\label{l.no_pp} Let $g \in \RR_{\UU(f)}$ and let $h \in \UU(f)$ such that
$h=g$ in $\FF^c(p_g,g)$.
Then there is no periodic point of $h$ in $\partial C_0(p_h,h)$.
\end{lemma}

\begin{proof}
Assume, by contradiction, that there exists a periodic point $q \in
\partial C_0(p_h,h)$, $q \in K$ a connected component of $\partial
C_0(p_h,h)$.
Since $h=g$ on $\FF^c(p_g,g) = \FF^c(p_h,h)$ we have that $q$ is a
periodic point of $g$ and hence $q$ is either elliptic or hyperbolic (for
$g_{/\FF^c(p_g,g)}$ and thus for $h_{/\FF^c(p_g,g)}$).
If $q$ is elliptic then we know that $C_0(q,h)$ is open and we get a
contradiction.

Assume that $q$ is hyperbolic. If $C_0(q,h)$ is open we are done.
If not, we know that $CW^s(q) \subset C_0(q,h)$ or $CW^u(q) \subset
C_0(q,h)$. For instance, assume that $CW^s(q) \subset C_0(q,h)$.
Since $h =g$ on $\FF^c(p_g,g)$ then every periodic point of $h$ in
$\FF^c(p_g,g)$ is elliptic nondegenerated or hyperbolic and there is
no saddle connections (since $g$ is $\KS$).
A theorem of J. Mather in \cite[Theorem~5.2]{Ma81} implies that
$CW^u(q) \subset \overline{CW^s(q)} \subset \overline{C_0(q,h)} \subset K$.
We know that there exists a continuous map $\gamma \colon B(q) \times [0,1]
\to \FF^c(p_g,g)$ such that $\gamma(q,\cdot) \subset
CW^s_{loc}(q,h)$, $\gamma(q,\cdot)$ is not constant and for every $z
\in B(q)$, $\gamma(z,t) \in C_0(z,h)$.
Therefore, for $z$ belonging to an appropriate component of $B(q) \setminus
(CW^s_{loc}(q) \cup CW^u_{loc}(q))$ we have that $\gamma(z, t_z) \in
CW^u_{loc}(q)$ for some $t_z$ and so $z \in K$.

\begin{figure}[phtb]
\centering
\psfrag{q}{$q$}
\psfrag{cs}{$CW^s_{loc}(q)$}
\psfrag{cu}{$CW^u_{loc}(q)$}
\psfrag{z}{$z$}
\psfrag{c0}{$C_0(q,h)$}
\psfrag{g}{$\gamma(z,t)$}
\includegraphics[height=5cm]{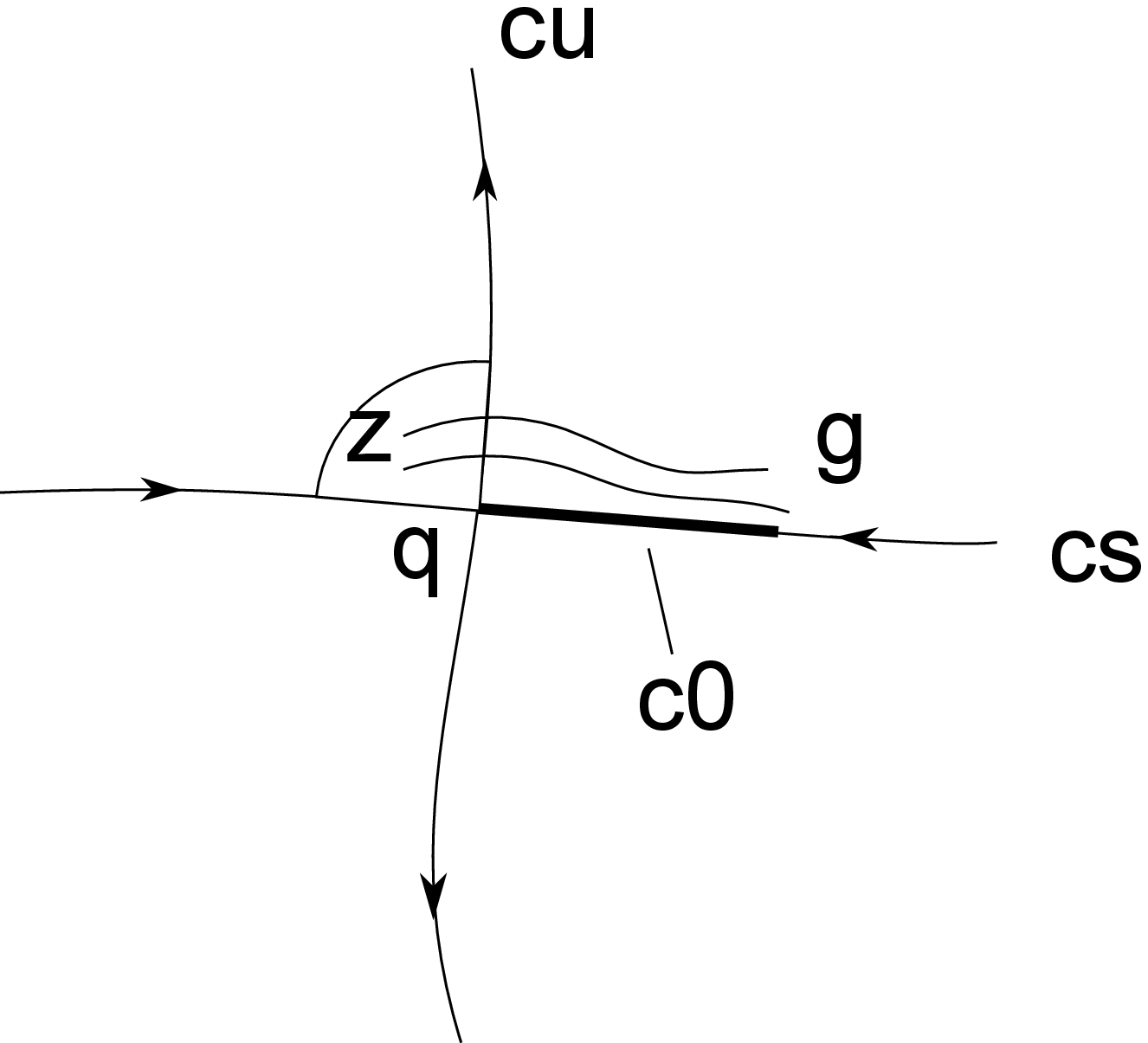}
\label{f.per}
\end{figure}

This implies that $K$ has nonempty interior, a contradiction.
\end{proof}

\begin{proposition}
\label{p.lamination}
Let $h \in \mathcal{U}(f)$ and let $K$ be a connected component of
$\partial C_0(p_h,h)$.
Then, the partition of $K$ by connected component of accessibility classes
form a $C^1$ lamination.
\end{proposition}

\begin{proof}
It is a direct consequence of Corollary~\ref{c.lamination}.
\end{proof}

We need a general result about $C^1$ lamination of subsets of the plane.

\begin{proposition}
\label{p.plane}
Let $K \subset \mathbb{R}^2$ be a compact and connected set with empty
interior and supporting a $C^1$ lamination.
Then
\begin{enumerate}
\item $\mathbb{R}^2 \setminus K$ has at least two connected components, and
\item if $\mathbb{R}^2 \setminus K$ has exactly two connected component  then $K \cong \mathbb{S}^1$.
\end{enumerate}
\end{proposition}

\begin{proof}
If $K$ contains a leaf which is diffeomorphic to a circle then clearly $K$
separates $\mathbb{R}^2$.
On the other hand, if $K$ does not contains such a leaf then by \cite{FO96},
$\mathbb{R}^2 \setminus K$ has at least four connected components, this
proves item 1.

Now, if $\mathbb{R}^2 \setminus K$ has exactly two connected components, by
the above it follows that $K$ contains a  leaf $W_0$ that is diffeomorphic
to a circle, which is unique otherwise the complement has at least three
connected components.
Arguing by contradiction, assume that there are other leaves of the
lamination than $W_0$.

Let $W(x)$ be the leaf of lamination through $x \in K$.
Orientate the leaf in an arbitrary way.
It follows that the $\alpha$ and $\omega$ limit set of the leaf must
contain $W_0$.
Otherwise, the result of \cite{FO96} applies and the complement of $K$ has
at least four connected components.

Consider a transversal section $J$ through $W_0$.
By the above, every point in $K \setminus J$ is in an arc of the lamination
having both ends in $J$.
The same arguments in the paper of \cite{FO96} yields that the lamination
could be extended to a foliation with singularities in the sphere having at
most one singularity of index $1$ and the others have index less than $1/2$.
It follows that the complement of $K$ has at least $3$ connected components,
a contradiction.
\end{proof}

We now state a theorem that will be useful in our context.

\begin{theorem}[Xia \cite{Xia06a}, Koropecki \cite{Ko10}]
\label{t.X,Ko}
Let $S$ be a compact surface without boundary and let $f \colon S \to S$
be a homeomorphism such that $\Omega(f) = S$.
Let $K$ be a compact connected invariant set.
Then, one of the following holds:
\begin{enumerate}
\item $K$ has a periodic point;
\item $K=S=\mathbb{T}^2$;
\item $K$ is an annular domain, i.e., there exists an open neighborhood
$U$ of $K$ homeomorphic to an annulus and $U \setminus K$ has
exactly two components (each one homeomorphic to an annulus).
\end{enumerate}
\end{theorem}

\begin{proposition}
Let $g \in \RR_{\UU(f)}$ and let $h \in \UU(f)$ such that $h=g$ on
$\FF^c(p_g,g)$.
Then, any connected component $K$ of $\partial C_0(p_h,h)$ is a simple
closed $C^1$-curve invariant for some power of $h$ (and $g$).
\end{proposition}

\begin{proof}
Let $K$ be a connected component of $\partial C_0(p_h,h)$.
By Proposition~\ref{p.lamination} we know that $K$ admits a $C^1$
lamination.
We have three possibilities:
\begin{itemize}
\item[$(1)$] $K \subset U$ where $U$ is homeomorphic to a disk and $K$
does not separate $U$;
\item[$(2)$] $K \subset U$ where $U$ is homeomorphic to a disk and $K$
does separate $U$;
\item[$(3)$] none of the above, i.e., in any neighborhood $U$ of $K$
we have non null-homotopic closed curves (in $\FF^c(p_h,h)$).
\end{itemize}

Proposition~\ref{p.plane} implies that $(1)$ cannot happen.
Lets consider situation $(2)$.
We consider an open set $U_0 \subset U$, where $U_0$ is a
connected component of the complement of $\overline{C_0(p_h,h)}$ and
$\partial U_0 \subset K$. Since $h|\FF^c(p_h,h)$ preserves the form
$\omega|\FF^c(p_h,h)$ we have for some integer $m$ that $h^m(U_0) =
U_0$.
This  implies that $h^m(K) \cap K \neq \emptyset$.
From the fact that $C_0(p_h,h)$ is invariant we get $h^m(K) = K$.
Since $K$ has no periodic point (from Lemma~\ref{l.no_pp}) and is not the
whole surface, we have from Theorem~\ref{t.X,Ko} that $K$ is an annular
domain and by Proposition~\ref{p.plane}, we have that $K$ is a simple
closed curve.

Finally, lets consider situation $(3)$. Notice that there are finitely
many components satisfying $(3)$.
On the other hand, $h$ maps a connected component $K$ satisfying $(3)$, to
another one also satisfying $(3)$.
Therefore, for some $m$ we have that $h^m(K)=K$, for any $K$ in $(3)$.
Applying Theorem~\ref{t.X,Ko} and Proposition~\ref{p.plane}, we get the
result as before.
\end{proof}

\begin{lemma}
Let $g \in \RR_{\UU(f)}$. Then $\overline{C_0(p_g,g)} = \FF^c(p_g,g)$.
\end{lemma}

\begin{proof}
Assume that this is not the case, and so, there is a connected component
$\CC$ of $\partial C_0(p_g,g)$ (which is a simple closed curve) and an
open annulus $U$ which is a neighborhood of $\CC$, such that  one
component of $U \setminus \CC \subset C_0(p_g,g)$ and the other one is
contained in the complement of $\overline{C_0(p_g,g)}$.
We know that $g^m(\CC) = \CC$ for some $m$.

Consider the family $\Gamma$ of $g^m$ invariant simple closed and
essential $C^1$-curve in $U$. Notice that curves in this family are
disjoint or coincide.
This is because, since $g$ is $\KS$, these curves cannot have rational
rotation number.
Now, if $\CC_1 \cap \CC_2 \neq \emptyset$, by invariance we have that they
intersects along the nonwandering set $\Omega(g^m|\CC_1) =
\Omega(g^m|\CC_2)$.
But if one (and hence both) are Denjoy maps there must exists a wandering
open set $U \subset \FF^c(p_g,g)$, a contradiction since $g$ preserves
area on the compact leaf $\FF^c(p_g,g).$

Let $V \subset \overline{V} \subset U$ be an annulus neighborhood of $\CC$
as in Lemma~\ref{l.segundo}.
Since $g \in \RR_{\UU(f)}$ and in particular $g$ is a continuity point of
the maps $\Gamma_2$ and $\Gamma_3$ (see \eqref{eq.gamma2}) it is not
difficult to see that there exists $\VV(g)$ such that if $h \in \VV(g)$
and $h=g$ on $\FF^c(p_g,g)$ then $\partial C_0(p_h,h)$ must have a
connected component in $V$ which must be an $h^m$-invariant (and so
$g^m$-invariant) essential simple closed $C^1$-curve.
By Lemma~\ref{l.segundo} we get a contradiction.
\end{proof}

Now we are ready to finish the proof of Theorem~\ref{t.accesibleRR} and
hence Theorem~\ref{main}.

\medskip

\noindent \textit{End of proof of Theorem~\ref{t.accesibleRR}:}
Let $g\in \RR_{\UU(f)}$ and we already know that $\overline{C_0(p_g,g)} =
\FF^c(p_g,g)$. We want to prove that $C_0(p_g,g)=\FF^c(p_g,g).$
We argue by contradiction and we assume that $C_0(p_g,g)\neq\FF^c(p_g,g)$
and let $\CC_i=\CC_i(g), i=1, \dots ,\ell$ be the connected components of
$\partial C_0(p_g,g).$
We know that every  $\CC_i$ is a simple closed $C^1$-curve non
null-homotopic invariant for $g^{m_i}$, for some $m_i$ and let $U_i$ be annulus
neighborhoods of $\CC_i.$

Since $g$ is a continuity point of $\Gamma_2$ and $\Gamma_3$ we get that
there exists a neighborhood $\VV(g)$ such that if $h\in\VV(g)$ then
\begin{itemize}
\item $\overline{C_0(p_h,h)}=\FF^c(p_h,h).$
\item $\FF^c(p_h,h) \setminus C_0(p_h,h)\cap U_i\neq\emptyset, i=1,\dots ,
\ell.$
\end{itemize}

Consider  the family of essential simple closed $C^1$-curves $g^m$-invariant
and contained in $U_i$ and let $V_i$ be as in Lemma~\ref{l.segundo}.
Since $g \in \RR_{\UU(f)}$ and for $\VV(g)$ as above we have  for any
$h \in \VV(g)$ and such that $h=g$ on $\FF^c(p_g,g)$  that
$\partial C_0(p_h,h)$ must have a connected components $\CC_i(h)$ (which are
simple closed curves) contained in every $V_i$.
By Lemma~\ref{l.segundo} this curves cannot be essential in $V_i$\,.
This implies that $\CC_i(h)$ must be null-homotopic.
And therefore $\overline{C_0(p_h,h)}\neq\FF^c(p_h,h),$ a contradiction.
\hfill $\Box$

\appendix
\section{Proof of Proposition~\ref{p.bb}}\label{a.apA}
\subsection{Bounded Variation}

Recall that $f \colon [a,b] \to \mathbb{R}$ is of {\em bounded
variation} if\;:
$$
\sup \left\{ \sum_{i=0}^{n-1} |f(x_{i+1}) - f(x_i)| \colon a=x_0 < x_1 <
\dots < x_n=b \right\} < \infty,
$$
and this supremum is denoted by $V(f;[a,b])$.

\begin{lemma}
\label{l.bv} Let $f \colon [a,b] \to \mathbb{R}$ of bounded
variation. We have the following:
\begin{enumerate}
\item if $[a_1,b_1] \subset [a,b]$ then  $V(f;[a_1,b_1]) \le V(f;[a,b])$,
\item if $(a_1,b_1)$ and $(a_2,b_2)$ are disjoint intervals contained in $[a,b]$
then $V(f;[a_1,b_1]) + V(f;[a_2,b_2]) \le V(f;[a,b])$. The same
holds for any finite disjoint collection of intervals $(a_i,b_i)$'s,
\item if $(a_1,b_1)$ and $(a_2,b_2)$ are disjoint intervals in $[a,b]$ and
$f([a_1,b_1]) \cup f([a_2,b_2]) \supset [c,d]$ then $V(f;[a_1,b_1])
+ V(f;[a_2,b_2]) \ge d - c$. A similar statement holds for any
finite disjoint collection of $(a_i,b_i)$'s,
\item if $f$ is the difference of two non-decreasing maps then $f$ is
of bounded variation.
\end{enumerate}
\end{lemma}

\begin{proof}
We just prove item 3, the others follows immediately from the
definition of bounded variation. Lets assume that $c \in
f([a_1,b_1])$. If $d \in f([a_1,b_1])$ then we are done. So, assume
that $d \notin f([a_1,b_1])$ and so $d \in f([a_2,b_2])$. Let $c^* =
\sup (f([a_1,b_1]))$ and $d^* = \inf (f([a_1,b_1]))$ then $c^* \ge
d^*$, $c^* \ge c$, and $d^* \le d$. Then $V(f;[a_1,b_1]) \ge c^* -
c$ and $V(f;[a_2,b_2]) \ge d- d^*$. Then,
$$
V(f;[a_1,b_1]) + V(f;[a_2,b_2]) \ge (d-d^*) + (c^* -c) \ge  d - c.
$$
By induction, we prove the statement for finite collections of
intervals.
\end{proof}

\subsection{Proof of Proposition~\ref{p.bb}}

Lets state it again for simplicity:

\textit{Let $\ell_1 \colon [-a,a] \to \mathbb{R}$ and $\ell_2 \colon [-b,b]
\to \mathbb{R}$ be two non-decreasing maps and let $\phi \colon [-b,b] \to
[-a,a]$ be also a non-decreasing map.
Then for any $\varepsilon > 0$ there exist $s,t$, $|s|,|t| \le \varepsilon$,
such that\,:
$$
\phi(\{ x\in [-b,b] \colon \ell_2(x)+t = x \} ) \cap \{ x\in [-a,a]
\colon \ell_1(x)+s = x \} = \emptyset.
$$
}

For a non-decreasing map $f \colon [a,b] \to \mathbb{R}$ and $y\in [a,b]$
we denote by $f_-(y)=\lim_{x\nearrow y} f(x)$ and $f_+(y) =
\lim_{x\searrow y} f(x)$.
We say that $f \colon [a,b]\to \mathbb{R}$ has a {\em jump} in $z\in (a,b)$
if $f_-(z) \neq f_+(z)$ (i.e., if $z$ is a discontinuity point of $f$).
Moreover, we say that $f$ has a {\em jump of size} $\varepsilon$ if $f$
has a jump in some $z$ such that $|f_-(z) - f_+(z)| \ge \varepsilon$.

\begin{lemma}
\label{l.aux}
Let $f \colon [a,b] \to \mathbb{R}$ be non-decreasing.
Then, for any $\varepsilon > 0$ there is $\delta > 0$ such that if
$0<y - x < \delta$ then either $f(y) - f(x) < \varepsilon$ or there exists
a jump of size $\varepsilon/2$ between $x$ and $y$.
\end{lemma}

\begin{proof}
By contradiction, suppose that there is $\varepsilon_0 > 0$ such that for
any $\delta>0$, there exist $x_\delta, y_\delta$ such that $0< y_\delta -
x_\delta < \delta$ and $f(y_\delta) - f(x_\delta) \ge \varepsilon_0$, and
there is no jump of size $\varepsilon_0/2$.

Let $(x_n)$ and $(y_n)$ be two sequences such that $0<y_n - x_n < 1/n$ and
$f(y_n) - f(x_n) \ge \varepsilon_0$, and there is no jump of size
$\varepsilon_0/2$ between $x_n$ and $y_n$, for every $n$.
Let $x$ be an accumulation point of $\{x_n\}.$
Since $f$ is non-decreasing, we may assume that $x_n$ approaches $x$ from
the left (otherwise $\lim f(y_n) - f(x_n)=f_+(x)-f_+(x)=0$).
We may assume then that $x_n \nearrow x.$
By the same argument, we have that $\{y_n\}$ has to approach $x$ from the
right, and we may assume that and so $y_n \searrow x$.
Thus, $f_+(x) - f_-(x) \ge \varepsilon_0$, which is a contradiction, since
$x_n \le x \le y_n$.
\end{proof}

Let $g_1 \colon [-a,a] \to \mathbb{R}$ defined by $g_1 = \ell_1(x) -x$
and $g_2 \colon [-b,b] \to \mathbb{R}$ defined by $g_2 = \ell_2(x) -x$.
Notice that, $g_1, g_2$ are of bounded variation and that
$$
\{ x\colon \ell_1(x) + s=x \} = g_1^{-1}(s) \quad \text{ and } \quad \{
x\colon \ell_2(x) + t=x \}= g_2^{-1}(s).
$$

\begin{lemma}
\label{l.ap1}
For any $s_1 < s_2$ the following hold:
\begin{enumerate}
\item $\overline{\phi^{-1}(g_1^{-1}(s_1))} \cap
\overline{\phi^{-1}(g_1^{-1}(s_2))}$ contains at most finitely many points.
\item For any $y$ in the above intersection there exists $\delta > 0$ such
that either
\begin{itemize}
\item[a)]$(y - \delta, y) \cap \phi^{-1}(g_1^{-1}(s_1)) = \emptyset
\text{ and } (y,y+\delta) \cap \phi^{-1}(g_1^{-1}(s_2)) = \emptyset$ or
\item[b)] $(y - \delta, y) \cap \phi^{-1}(g_1^{-1}(s_2)) = \emptyset
\text{ and } (y,y+\delta) \cap \phi^{-1}(g_1^{-1}(s_1)) = \emptyset.$
\end{itemize}
\end{enumerate}
\end{lemma}

\begin{proof}
Let $y \in \overline{\phi^{-1}(g_1^{-1}(s_1))} \cap
\overline{\phi^{-1}(g_1^{-1}(s_2))}$.
Observe that $\phi^{-1}(g_1^{-1}(s_1)) \cap \phi^{-1}(g_1^{-1}(s_2)) =
\emptyset$.
We claim that $y$ cannot be accumulated at one side (either right or left)
by both sets $\phi^{-1}(g_1^{-1}(s_1))$ and $\phi^{-1}(g_1^{-1}(s_2))$.
Otherwise, assume this for the left, let $x_n \nearrow y$, $x_n \in
\phi^{-1}(g_1^{-1}(s_1))$ and $z_n \nearrow y$, $z_n \in
\phi^{-1}(g_1^{-1}(s_2))$.
Then, $\phi(x_n) \nearrow \phi_-(y)$ and $\phi(z_n) \nearrow \phi_-(y)$.
Hence, $s_1 = g_1(\phi(x_n)) = \ell_1(\phi(x_n)) - \phi(x_n)$ and so $s_1=
(\ell_1)_-(\phi_-(y)) - \phi_-(y)$.
Analogously, $s_2 = g_1(\phi(z_n)) = \ell_1(\phi(z_n)) - \phi(z_n)$ and so
$s_2 = (\ell_1)_-(\phi_-(y)) - \phi_-(y)$, a contradiction since $s_1 \neq
s_2$.
This proves item 2.

To prove item 1, lets assume that for $y$ in the intersection situation
(a) holds.
Then $s_2 = (\ell_1)_-(\phi_-(y)) - \phi_-(y)$ and $s_1 =
(\ell_1)_+(\phi_+(y)) - \phi_+(y)$.
So, $(\ell_1)_+(\phi_+(y)) = s_1 + \phi_+(y)$ and $(\ell_1)_+(\phi_-(y)) =
s_2 + \phi_-(y)$.
Since $\phi_(y) \le \phi_+(y)$ and $\ell_1$ is non-decreasing then
$(\ell_1)_+(\phi_-(y)) \le (\ell_1)_+(\phi_+(y))$. Then, $s_1 + \phi_+(y)
\ge s_2 + \phi_-(y)$. Therefore,
$$
\phi_+(y) - \phi_-(y) \ge s_2 - s_1.
$$
This means that the jump of $\phi$ at $y$ is at least of size $s_2 - s_1$
and there are at most finitely many of them.
The proof is complete in this case.

Now, assume that (b) holds and so $s_1 = (\ell_1)_-(\phi_-(y))- \phi_-(y)$
and $s_2 = (\ell_1)_+(\phi_+(y)) - \phi_+(y)$.
Let $\varepsilon = s_2 - s_1$ and let $\delta$ $(< \varepsilon)$ from
Lemma~\ref{l.aux} applied to $\ell_1$.
Notices that
\begin{align*}
(\ell_1)_+(\phi_+(y)) - & (\ell_1)_-(\phi_-(y)) = s_2 + \phi_+(y) - s_1 -
\phi_-(y) \\
& = (s_2 - s_1) + \phi_+(y)) -\phi_-(y) \ge s_2 -s_1.
\end{align*}
If $y$ is a continuity point of $\phi$ then we have a $\ell_1$-jump of
size $s_2 - s_1$ at $\phi(y) =\phi_+(y) = \phi_-(y)$ and there can be just
finitely many of them.
On the other hand, there can be finitely many $y$'s such that the
$\phi$-jump at $y$ is at least $\delta$.
So, we just consider the set of $y$'s such that $\phi_+(y) - \phi_-(y) <
\delta$.
By Lemma~\ref{l.aux} there exists a $\ell_1$-jump in $[\phi_-(y),\phi_+(y)]$
of size at least $\varepsilon/2$.
For different $y$'s the intervals $[\phi_-(y), \phi_+(y)]$ are disjoint.
Since there are finitely many $\ell_1$-jumps of size at least
$\varepsilon/2$, we conclude that there are finitely many $y$'s in
$$
\overline{\phi^{-1}(g_1^{-1}(s_1))} \cap \overline{\phi^{-1}(g_1^{-1}(s_2))}
$$
and the lemma is proved.
\end{proof}

\medskip

Now we can prove Proposition~\ref{p.bb}.
\begin{proof}[Proof of Proposition~\ref{p.bb}]
Recall that $g_1(x) = \ell_1(x) - x$ and $g_2(x) = \ell_2(x) - x$.
Assume, by contradiction, that for some $\varepsilon > 0$ we have that for
any $s,t$, $|s|,|t| \le \varepsilon$ we have
$$
\phi(\{ x\in [-b,b] \colon \ell_2(x)+t = x \} ) \cap \{ x\in [-a,a] \colon
\ell_1(x)+s = x \} \neq \emptyset.
$$
We know that $g_2$ is of bounded variation, set $M = V(g_2;[-b,b])$ and
let $k$ be an integer, $k > M/(2\varepsilon)$.
Consider a partition of $[-\varepsilon,\varepsilon]$
$-\varepsilon \le s_1 < s_2 < \dots < s_k \le \varepsilon$, and let
$S_i = \overline{\phi^{-1}(g_1^{-1}(s_i))}$.
Notice that from our contradicting assumption that $g_2(S_i)\supset
[-\varepsilon,\varepsilon].$
From Lemma~\ref{l.ap1} we have that $S_i \cap S_j$ contains at most
finitely many points for $i\neq j$.
And if $i\neq j \neq l \neq i$ then $S_i\cap S_j\cap S_l = \emptyset$.

Let $y_1, \dots, y_m$ be the set of points that belongs to more than one
$S_i$.
For each $y_i$ let $\delta_i$ from Lemma~\ref{l.ap1} such that $(y_i -
\delta_i, y_i)$ intersects just one of the sets $S_j$, $j=1,\dots ,k$, and
the same for $(y_i, y_i + \delta_i)$.

Let $\displaystyle \widehat{S}_j = S_j \cap \left[ \bigcup_{i=1}^M
(y_i-\delta_i, y_i+ \delta_i)\right]^c$.
The sets $\widehat{S}_j$, $j=1,\dots, k$ are compact and disjoints.
For each $j$, choose $\displaystyle \widehat{U}_j = \bigcup_{l =1}^{m_j}
[a_l, b_l]$ such that $\widehat{S}_j \subset \widehat{U}_j$ and
$\widehat{U}_j \cap \widehat{U}_i = \emptyset$ if $j \neq i$.

Let
$$
U_j = \widehat{U}_j \cup \left( \bigcup_{i\colon S_j \cap [y_i -
\delta ,y_i] \neq \emptyset} [y_i - \delta_i, y_i] \right) \cup
\left( \bigcup_{i\colon S_j \cap [y_i,y_i + \delta] \neq \emptyset}
[y_i, y_i + \delta_i] \right).
$$
We can write $U_j$ as a union of finitely many compact and disjoint
intervals $I_j(1), \dots, I_j(m_j)$.

Now, we have\;:
\begin{itemize}
\item[$\cdot$] $g_2(U_j) \supset [-\varepsilon, \varepsilon]$ for any $j=1,
\dots, k$; and
\item[$\cdot$] $\interior (U_j) \cap \interior (U_l) = \emptyset$ if
$j \neq l$.
\end{itemize}

Therefore, we have
$$
\sum_{i=1}^m V(g_2;I_j(i)) \ge 2 \varepsilon,
$$
and so, from Lemma~\ref{l.bv}, we get
$$
V(g_2;[-b,b]) \ge \sum_{j=1}^k \sum_{i=1}^m V(g_2;I_j(i)) \ge k2
\varepsilon > M \ge V(g_2;[-b,b]),
$$
a contradiction.
This completes the proof.
\end{proof}


\begin{thebibliography}{RHRHTU07}

\bibitem[ACW15]{ACW}
A.~Avila, S.~Crovisier, and A.~Wilkinson.
\newblock Diffeomorphisms with positive metric entropy.
\newblock arXiv:1408.4252, 2015.

\bibitem[AM07]{AM07}
A.~Arbieto and C.~Matheus.
\newblock A pasting lemma and some applications for conservative systems.
\newblock {\em Ergodic Theory Dynam. Systems}, 27(5):1399--1417, 2007.
\newblock With an appendix by David Diica and Yakov Simpson-Weller.

\bibitem[Ano67]{An67}
D.~V. Anosov.
\newblock Geodesic flows on closed {R}iemannian manifolds of negative
  curvature.
\newblock {\em Trudy Mat. Inst. Steklov.}, 90:209, 1967.

\bibitem[BB]{BB}
D.~Bohnet and C.~Bonatti.
\newblock Partially hyperbolic diffeomorphisms with uniformly center foliation:
  the quotient dynamics.
\newblock To appear in Ergod. Th. {\&} Dynam. Sys.

\bibitem[BFP06]{BFP06}
J.~Bochi, B.~R. Fayad, and E.~Pujals.
\newblock A remark on conservative diffeomorphisms.
\newblock {\em C. R. Math. Acad. Sci. Paris}, 342(10):763--766, 2006.

\bibitem[BHH{\etalchar{+}}08]{BRRTU08}
K.~Burns, F.~Rodriguez Hertz, M.~A.~Rodriguez Hertz, A.~Talitskaya, and
  R.~Ures.
\newblock Density of accessibility for partially hyperbolic diffeomorphisms
  with one-dimensional center.
\newblock {\em Discrete Contin. Dyn. Syst.}, 22(1-2):75--88, 2008.

\bibitem[Boh13]{Boh13}
D.~Bohnet.
\newblock Codimension-1 partially hyperbolic diffeomorphisms with a uniformly
  compact center foliation.
\newblock {\em J. Mod. Dyn.}, 7(4):565--604, 2013.

\bibitem[BPSW01]{BPSW01}
K.~Burns, C.~Pugh, M.~Shub, and A.~Wilkinson.
\newblock Recent results about stable ergodicity.
\newblock In {\em Smooth ergodic theory and its applications (Seattle, WA,
  1999)}, volume~69 of {\em Proc. Sympos. Pure Math.}, pages 327--366. Amer.
  Math. Soc., Providence, RI, 2001.

\bibitem[BW99]{BW99}
K.~Burns and A.~Wilkinson.
\newblock Stable ergodicity of skew products.
\newblock {\em Ann. Sci. \'Ecole Norm. Sup. (4)}, 32(6):859--889, 1999.

\bibitem[BW08]{BW08}
K.~Burns and A.~Wilkinson.
\newblock Dynamical coherence and center bunching.
\newblock {\em Discrete Contin. Dyn. Syst.}, 22(1-2):89--100, 2008.

\bibitem[BW10]{BW10}
K.~Burns and A.~Wilkinson.
\newblock On the ergodicity of partially hyperbolic systems.
\newblock {\em Ann. of Math. (2)}, 171(1):451--489, 2010.

\bibitem[Car11]{CaP}
P.~D. Carrasco.
\newblock {\em Compact {D}ynamical {F}oliations}.
\newblock ProQuest LLC, Ann Arbor, MI, 2011.
\newblock Thesis (Ph.D.)--University of Toronto (Canada).

\bibitem[Cro14]{CrICM}
S.~Crovisier.
\newblock Dynamics of ${C}^1$-diffeomorphisms: global description and prospects
  for classification.
\newblock arXiv:1405.0305, 2014.

\bibitem[DM90]{DM90}
B.~Dacorogna and J.~Moser.
\newblock On a partial differential equation involving the {J}acobian
  determinant.
\newblock {\em Ann. Inst. H. Poincar\'e Anal. Non Lin\'eaire}, 7(1):1--26,
  1990.

\bibitem[DW03]{DW03}
D.~Dolgopyat and A.~Wilkinson.
\newblock Stable accessibility is {$C^1$} dense.
\newblock {\em Ast\'erisque}, (287):xvii, 33--60, 2003.
\newblock Geometric methods in dynamics. II.

\bibitem[FO96]{FO96}
R.~J. Fokkink and L.~G. Oversteegen.
\newblock The geometry of laminations.
\newblock {\em Fund. Math.}, 151(3):195--207, 1996.

\bibitem[Fra88]{Fra88}
J.~Franks.
\newblock Recurrence and fixed points of surface homeomorphisms.
\newblock {\em Ergodic Theory Dynam. Systems}, 8$^*$(Charles Conley Memorial
  Issue):99--107, 1988.

\bibitem[Gog11]{Go11}
A.~Gogolev.
\newblock Partially hyperbolic diffeomorphisms with compact center foliations.
\newblock {\em J. Mod. Dyn.}, 5(4):747--769, 2011.

\bibitem[GPS94]{GPS94}
M.~Grayson, C.~Pugh, and M.~Shub.
\newblock Stably ergodic diffeomorphisms.
\newblock {\em Ann. of Math. (2)}, 140(2):295--329, 1994.

\bibitem[Hop39]{Hop39}
E.~Hopf.
\newblock Statistik der geod\"atischen {L}inien in {M}annigfaltigkeiten
  negativer {K}r\"ummung.
\newblock {\em Ber. Verh. S\"achs. Akad. Wiss. Leipzig}, 91:261--304, 1939.

\bibitem[HPS77]{HPS77}
M.~Hirsch, C.~Pugh, and M.~Shub.
\newblock {\em Invariant manifolds}, volume 583 of {\em Lect. Notes in Math.}
\newblock Springer Verlag, 1977.

\bibitem[Kor10]{Ko10}
A.~Koropecki.
\newblock Aperiodic invariant continua for surface homeomorphisms.
\newblock {\em Math. Z.}, 266(1):229--236, 2010.

\bibitem[Lew80]{Le80}
J.~Lewowicz.
\newblock Lyapunov functions and topological stability.
\newblock {\em J. Differential Equations}, 38(2):192--209, 1980.

\bibitem[Mat81]{Ma81}
J.~Mather.
\newblock Invariant subsets for area preserving homeomorphisms of surfaces.
\newblock In {\em Mathematical analysis and applications, {P}art {B}}, volume~7
  of {\em Adv. in Math. Suppl. Stud.}, pages 531--562. Academic Press, New
  York, 1981.

\bibitem[Mos65]{Mo65}
J.~Moser.
\newblock On the volume elements on a manifold.
\newblock {\em Trans. Amer. Math. Soc.}, 120:286--294, 1965.

\bibitem[PS96]{PS97m}
C.~Pugh and M.~Shub.
\newblock Stable ergodicity and partial hyperbolicity.
\newblock In {\em International {C}onference on {D}ynamical {S}ystems
  ({M}ontevideo, 1995)}, volume 362 of {\em Pitman Res. Notes Math. Ser.},
  pages 182--187. Longman, Harlow, 1996.

\bibitem[PS97]{PS97}
C.~Pugh and M.~Shub.
\newblock Stably ergodic dynamical systems and partial hyperbolicity.
\newblock {\em J. Complexity}, 13(1):125--179, 1997.

\bibitem[PS00]{PS00}
C.~Pugh and M.~Shub.
\newblock Stable ergodicity and julienne quasi-conformality.
\newblock {\em J. Eur. Math. Soc. (JEMS)}, 2(1):1--52, 2000.

\bibitem[PSW97]{PSW97}
C.~Pugh, M.~Shub, and A.~Wilkinson.
\newblock H{\"o}lder foliations.
\newblock {\em Duke Math. J.}, 86:517--546, 1997.

\bibitem[RH05]{RH05}
F.~Rodriguez~Hertz.
\newblock Stable ergodicity of certain linear automorphisms of the torus.
\newblock {\em Ann. of Math. (2)}, 162(1):65--107, 2005.

\bibitem[RHRHTU07]{RRTU07}
F.~Rodriguez~Hertz, M.~A. Rodriguez~Hertz, A.~Tahzibi, and R.~Ures.
\newblock A criterion for ergodicity of non-uniformly hyperbolic
  diffeomorphisms.
\newblock {\em Electron. Res. Announc. Math. Sci.}, 14:74--81, 2007.

\bibitem[RHRHU07]{RRU07}
F.~Rodriguez~Hertz, M.~A. Rodriguez~Hertz, and R.~Ures.
\newblock A survey of partially hyperbolic dynamics.
\newblock In {\em Partially hyperbolic dynamics, laminations, and
  {T}eichm\"uller flow}, volume~51 of {\em Fields Inst. Commun.}, pages 35--87.
  Amer. Math. Soc., Providence, RI, 2007.

\bibitem[RHRHU08]{RRU08}
F.~Rodriguez~Hertz, M.~A. Rodriguez~Hertz, and R.~Ures.
\newblock Accessibility and stable ergodicity for partially hyperbolic
  diffeomorphisms with 1{D}-center bundle.
\newblock {\em Invent. Math.}, 172(2):353--381, 2008.

\bibitem[RS{\v{S}}96]{RSS96}
D.~Repov{\v{s}}, A.~B. Skopenkov, and E.~V. {\v{S}}{\v{c}}epin.
\newblock {$C^1$}-homogeneous compacta in {${\bf R}^n$} are
  {$C^1$}-submanifolds of {${\bf R}^n$}.
\newblock {\em Proc. Amer. Math. Soc.}, 124(4):1219--1226, 1996.

\bibitem[Shu71]{Sh71}
M.~Shub.
\newblock Topologically transitive diffeomorphisms on ${T}^4$.
\newblock volume 206 of {\em Springer Lecture Notes in Mathematics}, pages
  39--40. Springer Verlag, 1971.

\bibitem[SW00]{SW00a}
M.~Shub and A.~Wilkinson.
\newblock Stably ergodic approximation: two examples.
\newblock {\em Ergodic Theory Dynam. Systems}, 20(3):875--893, 2000.

\bibitem[Wei71]{We71}
A.~Weinstein.
\newblock Symplectic manifolds and their {L}agrangian submanifolds.
\newblock {\em Advances in Math.}, 6:329--346 (1971), 1971.

\bibitem[Wil10]{Wi10}
A.~Wilkinson.
\newblock Conservative partially hyperbolic dynamics.
\newblock In {\em Proceedings of the {I}nternational {C}ongress of
  {M}athematicians. {V}olume {III}}, pages 1816--1836. Hindustan Book Agency,
  New Delhi, 2010.

\bibitem[Xia06]{Xia06a}
Z.~Xia.
\newblock Area-preserving surface diffeomorphisms.
\newblock {\em Comm. Math. Phys.}, 263(3):723--735, 2006.

\bibitem[XZ06]{Xia06}
Z.~Xia and H.~Zhang.
\newblock A {$C\sp r$} closing lemma for a class of symplectic diffeomorphisms.
\newblock {\em Nonlinearity}, 19(2):511--516, 2006.

\bibitem[Zeh73]{Z73}
E.~Zehnder.
\newblock Homoclinic points near elliptic fixed points.
\newblock {\em Comm. Pure Appl. Math.}, 26:131--182, 1973.

\bibitem[Zha15]{Zha15}
Z.~Zhang.
\newblock ${C}^r$ density of stable ergodicity for a class of partially
  hyperbolic systems.
\newblock arXiv:1507.03556, 2015.

\end{thebibliography}

\newcommand{\etalchar}[1]{$^{#1}$}

\end{document}